\documentclass[pdflatex,sn-mathphys-num]{sn-jnl}

\usepackage{graphicx}%
\usepackage{multirow}%
\usepackage{amsmath,amssymb,amsfonts}%
\usepackage{amsthm}%
\usepackage{mathrsfs}%
\usepackage[title]{appendix}%
\usepackage{xcolor}%
\usepackage{textcomp}%
\usepackage{manyfoot}%
\usepackage{booktabs}%
\usepackage{algorithm}%
\usepackage{algorithmicx}%
\usepackage{algpseudocode}%
\usepackage{listings}%
\usepackage{graphicx}%
\usepackage{multirow}%
\usepackage{amsmath,amssymb,amsfonts}%
\usepackage{mathrsfs}%
\usepackage[title]{appendix}%
\usepackage{xcolor}%
\usepackage{textcomp}%
\usepackage{manyfoot}%
\usepackage{booktabs}%
\usepackage{algorithm}%
\usepackage{algorithmicx}%
\usepackage{algpseudocode}%
\usepackage{listings}%
\usepackage{url}

\usepackage{hyperref}
\usepackage{tabularx,booktabs}
\newcolumntype{Y}{>{\centering\arraybackslash}X}
\usepackage[table]{xcolor}
\usepackage{makecell}
\usepackage{blkarray}
\usepackage{nicematrix}
\usepackage{rotating}
\usepackage{multirow}
\DeclareMathAlphabet{\pazocal}{OMS}{zplm}{m}{n}
\SetMathAlphabet\pazocal{bold}{OMS}{zplm}{bx}{n}

\usepackage{subcaption}
\usepackage{sidecap}
 \usepackage{wrapfig}
 \usepackage{tabularx,booktabs}
\newcolumntype{Y}{>{\centering\arraybackslash}X}
\usepackage{booktabs}
\usepackage{tikz}
\usepackage{blkarray}
\usepackage{colortbl}
\usepackage{multirow}
\usepackage{subcaption}
\usepackage{xcolor}
\usepackage{graphicx}
\usepackage{float}

\theoremstyle{thmstyleone}%
\newtheorem{theorem}{Theorem}
\newtheorem{lemma}[theorem]{Lemma}
\newtheorem{corollary}[theorem]{Corollary}

\theoremstyle{thmstyletwo}%

\theoremstyle{thmstylethree}%
\newtheorem{definition}{Definition}%

\begin{document}

\title[Random Walks on Virtual Persistence Diagrams]{Random Walks on Virtual Persistence Diagrams}


\author*{\fnm{Charles} \sur{Fanning}}\email{cfannin8@students.kennesaw.edu}

\author{\fnm{Mehmet} \sur{Aktas}}\email{maktas1@kennesaw.edu}

\affil{\orgdiv{School of Data Science and Analytics}, 
\orgname{Kennesaw State University}, 
\orgaddress{\street{1000 Chastain Rd NW}, 
\city{Kennesaw}, 
\postcode{30144}, 
\state{Georgia}, 
\country{United States}}}


\abstract{
In the uniformly discrete case of virtual persistence diagram groups $K(X,A)$, we construct a translation-invariant heat semigroup. The kernels are supported on a countable subgroup $H$, and the restriction to $H$ has Fourier exponent $\lambda_H$ satisfying $\lambda_H(\theta)=\sum_{\kappa\in H\setminus\{0\}}\bigl(1-\Re\,\theta(\kappa)\bigr)\,\nu(\kappa),$ for a symmetric $\nu\in\ell^1(H\setminus\{0\})$. This gives a symmetric jump process on $H$. The exponent $\lambda_H$ determines heat kernels, which define reproducing kernel Hilbert spaces and their associated semimetrics. Convex orders on the mixing measures give monotonicity for the kernels, Hilbert spaces, and semimetrics.
}



\maketitle

\section{Introduction}\label{sec1}

Persistence diagrams give stable and widely used summaries of filtered topological data \cite{892133,Zomorodian2005ComputingPH,Oudot2015PersistenceT}, and many problems in topological data analysis require vectorizations, kernels, and other analytic representations built from them. Existing work studies persistence diagrams through vectorizations and kernel methods in linear spaces. Algebraic approaches to persistence show that diagrams also come from M\"obius inversion of rank-type data. Signed multiplicities occur intrinsically in these constructions, and the resulting diagram objects extend beyond interval decompositions \cite{Patel_2018,Kim2021GeneralizedPD}. Virtual persistence diagrams address this algebraic structure by passing from the commutative monoid of persistence diagrams to an abelian group \cite{bubenik2022virtualpd}. That abelian group carries characters, Fourier transforms, and semigroup methods as intrinsic tools. In this paper we develop those harmonic-analytic tools for virtual persistence diagrams in infinite rank.

\noindent\textbf{\underline{Objective}}

The first objective is to construct a translation-invariant heat semigroup $(P_t)_{t\ge0}$ on the virtual persistence diagram group and to show a countable subgroup $H$ on which the semigroup is supported and has a L\'evy--Khintchine exponent $\lambda_H$.

The second objective is to use this representation to define kernels, semimetrics, and invariants and to establish inequalities and order relations among them.

\noindent\textbf{\underline{Main Results}}


The finite theory is described by graph-Laplacian heat semigroups on groups of the form $K(X,A)\cong\mathbb Z^{|X\setminus A|}$. The infinite-rank construction proceeds from a symmetric translation-invariant pair-jump kernel satisfying the summability condition \eqref{eq:psi-summability}, which defines a global negative-definite symbol $\lambda^\psi$ on $\widehat{K(X,A)}$. Finite-rank symbols come from $\lambda^\psi$ by projection and include the boundary contribution determined by jumps leaving the finite set. The symbol $\lambda^\psi$ generates a symmetric, translation-invariant, strongly continuous Markov semigroup $(P_t)_{t\ge0}$ on $K(X,A)$ with convolution kernels $p_t^\psi=P_t\delta_0$.

The kernels $p_t^\psi$ generate a countable subgroup
\begin{equation}\label{eq:intro-H}
H
=
\Big\langle \bigcup_{q\in\mathbb Q_{\ge0}} \operatorname{supp}(p_q^\psi) \Big\rangle
\subset K(X,A),
\end{equation}
and satisfy $\operatorname{supp}(p_t^\psi)\subset H$ for all $t\ge0$. The semigroup acts through $H$, and its restriction to $\ell^2(H)$ is a convolution semigroup.

The reduction to $H$ gives the Fourier representation
\[
\widehat p_t^\psi(\theta)=e^{-t\lambda_H(\theta)}
\]
for $\theta\in\widehat H$ and $t\ge0$, where $\lambda_H:\widehat H\to[0,\infty)$. The first main theorem gives a L\'{e}vy--Khintchine representation for $\lambda_H$.

\begin{theorem}
\label{thm:intro-lk}
There exists, for all $\theta\in\widehat H$, a unique symmetric $\ell^1(H\setminus\{0\})\ni \nu\colon H\setminus\{0\}\to[0,\infty)$ such that
\begin{equation}
\lambda_H(\theta)
=
\sum_{\kappa\in H\setminus\{0\}}
\bigl(1-\Re\,\theta(\kappa)\bigr)\,\nu(\kappa)
\end{equation}
\end{theorem}

The theorem identifies $(p_t^\psi)$ as a symmetric jump process on $H$ with jump intensity $\nu$ and Fourier exponent $\lambda_H$. In particular, the heat semigroup defines a random walk on virtual persistence diagrams with state space $H$.

The L\'{e}vy--Khintchine representation induces a functional calculus on $\lambda_H$. Let $\eta$ be a finite positive Borel measure on $[0,\infty)$ and define
\[
m_\eta(\lambda)
=
\int_{[0,\infty)} e^{-u\lambda}\,d\eta(u)
\]
for $\lambda\ge0$. Define the translation-invariant kernel
\[
K_\eta(g,h)
=
\int_{\widehat H}
\theta(h-g)\,m_\eta(\lambda_H(\theta))\,d\mu_{\widehat H}(\theta)
\]
for $g,h\in H$, with reproducing kernel Hilbert space $\mathcal H_{K_\eta}$ and semimetric
\[
d_\eta(g,h)^2
=
K_\eta(g,g)+K_\eta(h,h)-2\Re K_\eta(g,h).
\]
We also define
\begin{align*}
A_\eta &= \int_{\widehat H}\lambda_H(\theta)\,m_\eta(\lambda_H(\theta))\,d\mu_{\widehat H}(\theta), \\
B_\eta &= K_\eta(0,0).
\end{align*}
whenever the integral defining $A_\eta$ is finite.

The function $m_\eta$ represents a mixture of heat scales, and the kernel $K_\eta$ is the corresponding mixture of heat kernels. This induces an order on the kernel family given by convex order on the mixing measures $\eta$.

\begin{theorem}
Let $\eta_1,\eta_2$ be finite positive Borel measures on $[0,\infty)$ with finite first moments. If $\eta_1 \preceq_{\mathrm{cx}} \eta_2$, then:
\begin{enumerate}
\item $K_{\eta_1}\preceq K_{\eta_2}$. In particular:
\begin{enumerate}
\item $\mathcal H_{K_{\eta_1}}\hookrightarrow \mathcal H_{K_{\eta_2}}$ contractively.
\item $d_{\eta_1}(g,h)\le d_{\eta_2}(g,h)$ for all $g,h\in H$.
\item $B_{\eta_1}\le B_{\eta_2}$.
\end{enumerate}
\item If, in addition, $A_{\eta_2}<\infty$, then:
\begin{enumerate}
\item $A_{\eta_1}<\infty$.
\item $A_{\eta_1}\le A_{\eta_2}$.
\item $d_{\eta_2}(g,h)\le A_{\eta_2}^{1/2}\rho(g,h)$ for all $g,h\in H$.
\end{enumerate}
\end{enumerate}
\end{theorem}


The first theorem shows $(p_t^\psi)$ as a symmetric jump process on $H$ with jump intensity $\nu$ and Fourier exponent $\lambda_H$. Its Fourier transform satisfies \( \widehat p_t^\psi(\theta)=e^{-t\lambda_H(\theta)}. \) The second theorem proves that the kernels $K_\eta$ are ordered by convex order on the mixing measures $\eta$. This ordering transfers directly to the reproducing kernel Hilbert spaces, the semimetrics $d_\eta$, and the quantities $A_\eta$ and $B_\eta$. These results connect the jump-process structure of $(p_t^\psi)$ on $H$ with a family of translation-invariant kernels determined by the heat-scale mixtures $m_\eta(\lambda_H)$.

\noindent\textbf{\underline{New Difficulties and Limitations}}

The infinite uniformly discrete case introduces two obstructions. First, semimetrics and metrics obtained from functional calculus applied to $\lambda_H$ do not in general recover the transport metric $\rho$. Second, convex order on the representing measures $\eta$ on $[0,\infty)$ gives the relevant order structure, rather than any order on the group $H$. Consequently, one should not expect reverse comparison with $\rho$ or majorization principles formulated directly on $H$ without additional structure.

The transport metric $\rho$ is induced by optimal matching and restricts to $H$. By contrast, the semimetrics $d_\eta$, the Green semimetric $d_{G_s}$, and the semimetrics obtained from multipliers of $\lambda_H$ are defined through functional calculus applied to the symbol $\lambda_H$. These semimetrics depend on the chosen semigroup and are not determined by the transport geometry alone. The random walk generated by $(p_t^\psi)$ can undersample directions that are large with respect to $\rho$.

The semimetrics obtained from $\lambda_H$ satisfy upper bounds in terms of $\rho$, but no reverse bounds hold in general. The upper bounds follow from estimating $|\theta(\gamma)-1|$ in terms of the displacement $\rho(\gamma,0)$. No converse inequality holds in general, since the spectral weights defining the semimetrics derived from $\lambda_H$ may concentrate on characters with small oscillation.

The kernels $K_\eta$ are parameterized by finite measures $\eta$ on $[0,\infty)$ through the functional calculus $m_\eta(\lambda_H)$. The relevant order is convex order on the measures $\eta$, not an order on the group $H$. The monotonicity theorem therefore acts on the parameter $\eta$ and does not by itself induce an order on $H$ or a monotonicity principle for probability measures on $H$.

One should not expect the Green semimetric $d_{G_s}$ to satisfy a lower bound of the form $d_{G_s}\ge c\,\rho$. The Green semimetric is defined by a spectral weight built from $\lambda_H$, and the factor $|\theta(\gamma)-1|^2$ may be small on large sets of characters even when $\rho(\gamma,0)$ is large. Convex-order monotonicity defines a majorization relation on the representing measures $\eta$, and this relation induces monotonicity of the kernels $K_\eta$, the associated reproducing kernel Hilbert spaces, and the semimetrics $d_\eta$.

\subsection{Our Contributions}

We construct an infinite-rank heat semigroup on the virtual persistence diagram group from consistent finite-rank graph-Laplacian semigroups under the standing assumption (Definition~\ref{def:standing-assumption}, Lemma~\ref{lem:projection}, and Theorem~\ref{thm:global}) and show a countable subgroup $H$ that supports all convolution kernels and Fourier transforms (Lemma~\ref{lem:effective-support}). We prove that this semigroup has a Fourier representation on $H$ with symbol given by a L\'evy--Khintchine formula (Theorem~\ref{thm:lk-H}) and show the corresponding jump intensity, which defines a symmetric random walk on virtual persistence diagrams. We then use this symbol to define a functional calculus that constructs heat-scale mixture kernels $K_\eta$, their Hilbert spaces $\mathcal H_{K_\eta}$, semimetrics $d_\eta$, and invariants $A_\eta$ and $B_\eta$. Finally, we prove that convex order on the mixing measures determines the ordering of these kernels, Hilbert spaces, semimetrics, and invariants, and we derive Lipschitz and Sobolev estimates, mass tail and covering bounds for the random walk, and truncation approximations.

\subsection{Organization of the Paper}

\begin{itemize}
\item Section~\ref{sec:background} introduces persistent homology, persistence diagrams, virtual persistence diagram groups $K(X,A)$, and Wasserstein geometry.
\item Section~\ref{sec:heat-rkhs} constructs the infinite-rank heat semigroup, identifies the subgroup $H$, and develops the Fourier and L\'evy--Khintchine representations.
\item Section~\ref{sec:stoch-mc} develops the random walk on $H$ and proves majorization of the kernel family and associated invariants under convex order.
\item Section~\ref{sec:example} presents a finite weighted-graph example illustrating the heat flow, random walk, and invariants in the discrete case.
\end{itemize}

\section{Background and Notation}
\label{sec:background}

\noindent\textbf{\underline{Historical Context}}


Persistent homology begins with a filtration
\[
\emptyset = K^0 \subseteq K^1 \subseteq \cdots \subseteq K^m = K,
\]
and studies the evolution of the homology groups \(H_k(K^i)\) along the filtration. For \(p \ge 0\), the \(p\)-persistent \(k\)-th homology is defined by
\[
H_k^{i,p} = \operatorname{im}\bigl(H_k(K^i)\to H_k(K^{i+p})\bigr),
\]
so that \(H_k^{i,p}\) consists of those \(k\)-dimensional homology classes that are present at stage \(i\) and persist until at least stage \(i+p\). In this way, persistent homology records the birth and death of topological features, represented algebraically by homology classes, across the filtration. In the one-parameter case, these birth--death intervals determine barcodes and persistence diagrams, which give canonical summaries of filtered topological and geometric data \cite{892133,Zomorodian2005ComputingPH,Oudot2015PersistenceT}.

A persistence diagram $D_k$ is a finite multiset of points in
\[
\overline{\mathbb{R}}^2_{<} = \{(b,d)\in\mathbb{R}\times(\mathbb{R}\cup\{\infty\}) : b<d\}.
\]
The diagonal \( \Delta = \{(x,x):x\in\mathbb{R}\} \) is included with infinite multiplicity. Points $(b,\infty)$ are included. The Wasserstein metrics are defined by optimal transport with the diagonal $\Delta$ as a sink for unmatched mass. Persistence diagrams form a metric space with a distinguished subset $\Delta$.

In algebraic approaches to persistence, diagram constructions come from M\"obius inversion of rank-type data, which produces signed multiplicities and does not require interval decompositions \cite{Patel_2018, Kim2021GeneralizedPD}. Ordinary persistence diagrams form a commutative monoid under pointwise summation, but this monoid is not closed under the diagram constructions produced by M\"obius inversion, since those constructions naturally give signed multiplicities. This lack of closure forces passage to formal differences and hence to a group completion. 

Throughout this paper, $(X,d,A)$ denotes a metric pair, where $d\colon X\times X\to[0,\infty)$ is a metric and $A\subseteq X$ is a distinguished subset, referred to as the diagonal. Let $X/A$ denote the quotient space obtained by collapsing $A$ to a single point, and write $[A]\in X/A$ for the resulting basepoint. We freely identify $(X,d,A)$ with the pointed metric space $(X/A,\overline d,[A])$, where $\overline d$ is the quotient metric induced by $d$, following \cite{bubenik2022virtualpd}.

\subsection{Virtual Persistence Diagrams}

Assume that $A\neq\varnothing$. For $x\in X$, define $d(x,A)=\inf_{a\in A} d(x,a)$ and define the $1$--strengthened metric on $X$ by
\[
d_1(x,y)=\min\bigl(d(x,y),\, d(x,A)+d(y,A)\bigr).
\]
As shown in \cite{bubenik2022virtualpd}, $d_1$ descends to a genuine metric $\overline d_1$ on $X/A$, which is fixed throughout as the ground metric for all Wasserstein and analytic constructions.

Let $D(X)$ denote the free commutative monoid on $X$, realized as the set of finitely supported functions $f\colon X\to\mathbb N$ with pointwise addition, and let $D(A)\subseteq D(X)$ be the submonoid of functions supported on $A$. The quotient commutative monoid
\[
D(X,A)=D(X)/D(A)
\]
is the monoid of finite persistence diagrams on $(X,d,A)$, with neutral element denoted by $0$. Its Grothendieck group completion is denoted by $K(X,A)$.

Let $W_1$ denote the $1$--Wasserstein metric on $D(X,A)$ induced by the ground metric $\overline d_1$ on $X/A$, as defined in \cite{bubenik2022virtualpd}. For $p=1$, $W_1$ is translation invariant. This invariance implies that the formula
\[
\rho(\alpha-\beta,\gamma-\delta) = W_1(\alpha+\delta,\gamma+\beta), \qquad \alpha,\beta,\gamma,\delta\in D(X,A),
\]
is well defined. We refer to $(K(X,A),\rho)$ as the virtual persistence diagram group equipped with its $1$--Wasserstein metric \cite{bubenik2022virtualpd}.

\begin{figure}[h!]
\captionsetup[subfigure]{justification=centering}
 \centering
 \begin{subfigure}[h]{0.4\textwidth}
 \centering
 \includegraphics[width=\textwidth]{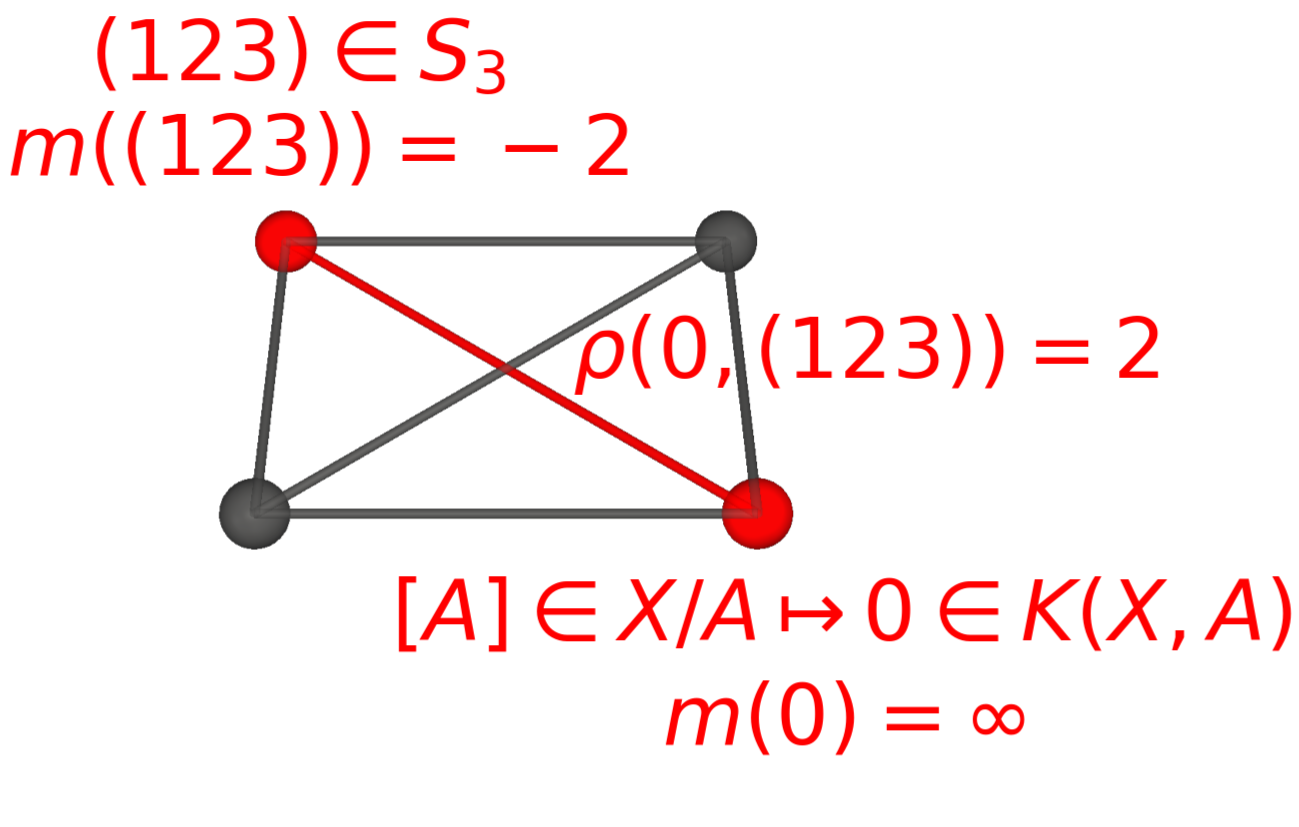}%
 \caption{Finite case }
 \end{subfigure}%
 ~ 
 \begin{subfigure}[h]{0.4\textwidth}
 \centering
 \includegraphics[width=\textwidth]{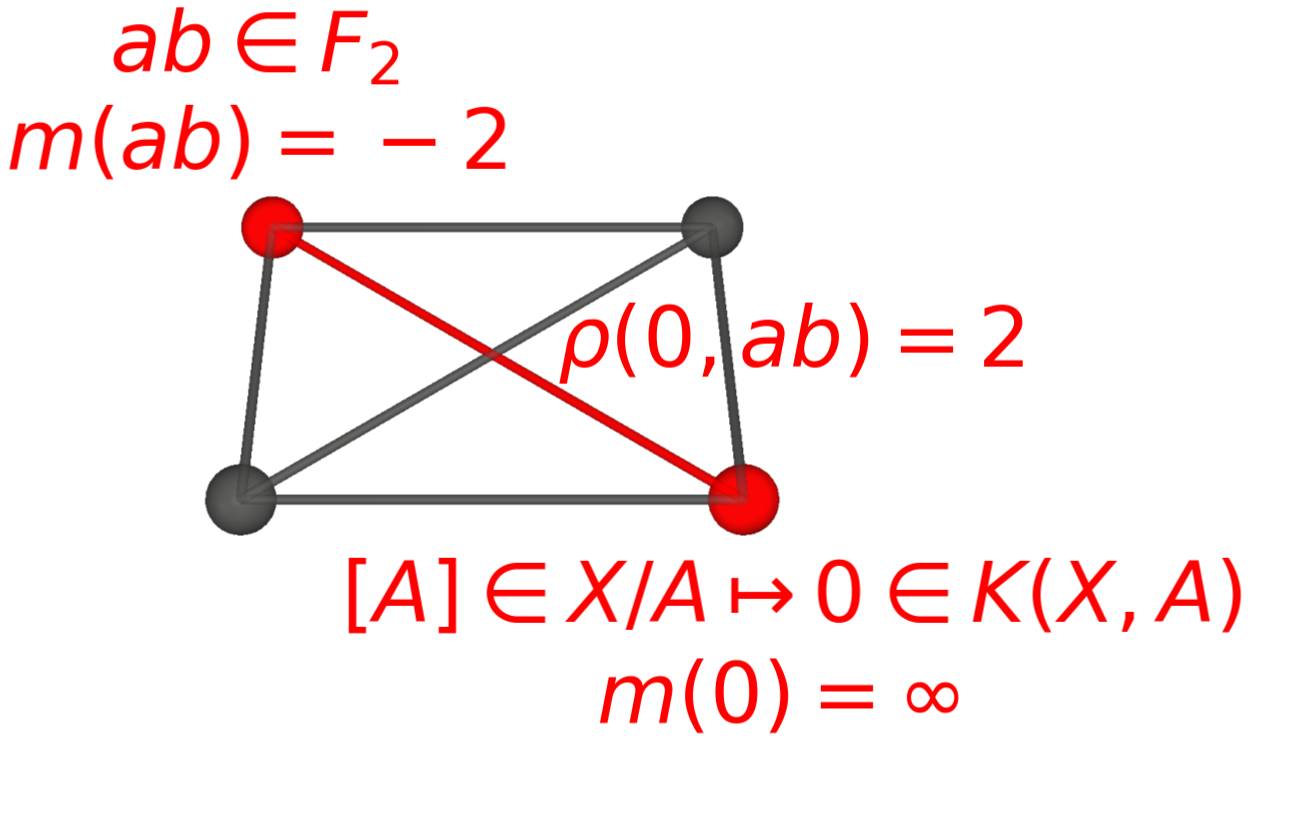}%
 \caption{ Discrete case.}
 \end{subfigure}%
 \caption{\emph{(a)} A finite virtual persistence diagram on the Cayley graph of $S_3$ with its word metric, with basepoint $0\in K(X,A)$ and element $(123)\in S_3$ at distance $\rho(0,(123))=2$.
\emph{(b)} An infinite uniformly discrete virtual persistence diagram modeled on the Cayley graph of the free group $F_2=\langle a,b\rangle$, with element $ab\in F_2$ at distance $\rho(0,ab)=2$.}
 \label{fig:finite-vs-discrete-vpd}
\end{figure}

Figure~\ref{fig:finite-vs-discrete-vpd} contrasts two virtual persistence diagrams with identical support but embedded in different metric spaces. In panel~(a), the diagram is supported in the Cayley graph of the finite group $S_3$, and the associated virtual persistence diagram group $K(X,A)$ is finite. In panel~(b), the diagram is supported in the Cayley graph of the free group $F_2$, which is infinite and uniformly discrete, and in this case $K(X,A)$ is an infinite discrete abelian group.

\subsection{Reproducing Kernel Hilbert Spaces for Virtual Persistence Diagrams}

\subsubsection{Finite virtual persistence diagrams}

Assume that $X\setminus A$ is finite. Then $K(X,A)$ is the free abelian group on $X\setminus A$, which we identify with $\mathbb Z^{|X\setminus A|}$ after fixing an ordering of $X\setminus A$. The Pontryagin dual of $K(X,A)$ is therefore $\widehat K\cong\mathbb T^{|X\setminus A|}$, equipped with normalized Haar probability measure $\mu$, and for $\theta\in\widehat K$ and $\gamma\in K(X,A)$ we write the corresponding character as $\chi_\theta(\gamma)=e^{i\langle
\gamma,\theta\rangle}$.

For each $\theta=(\theta_1,\dots,\theta_{|X\setminus A|})\in\mathbb T^{|X\setminus A|}$, with coordinates taken with respect to the fixed ordering of $X\setminus A$ used above, define the \emph{phase function}
\[
\phi_\theta\colon X/A\longrightarrow \mathbb R/2\pi\mathbb Z,
\qquad
\phi_\theta([A])=0,
\quad
\phi_\theta(x_j)=\theta_j\ (\mathrm{mod}\ 2\pi).
\]
The Lipschitz seminorm $\mathrm{Lip}_{\overline d_1}(\phi_\theta)$ is taken with respect to the metric $\overline d_1$ on $X/A$ and the geodesic metric on $\mathbb R/2\pi\mathbb Z$.

\begin{lemma}[{\cite[Lemma~3]{fanning2025reproducingkernelhilbertspaces}}]
\label{lem:phase-vs-lip}
For every $\theta\in\mathbb T^{|X\setminus A|}$,
\[
\frac{2}{\pi}\,
\mathrm{Lip}_{\overline d_1}(\phi_\theta)
\ \le\
\mathrm{Lip}_\rho(\chi_\theta)
\ \le\
\mathrm{Lip}_{\overline d_1}(\phi_\theta).
\]
\end{lemma}

Let $w_{\min}$ and $w_{\max}$ denote the minimal and maximal nonzero edge weights, and let $d_{\min}$ and $d_{\max}$ denote the minimal and maximal nonzero edge lengths, in the weighted graph model of $(X/A,\overline d_1)$ used in \cite{fanning2025reproducingkernelhilbertspaces}. Let $\lambda(\theta)$ denote the corresponding Laplacian symbol on $\mathbb T^{|X\setminus A|}$.

\begin{lemma}[{\cite[Lemma~4]{fanning2025reproducingkernelhilbertspaces}}]
\label{lem:lambda-vs-L}
For every $\theta\in\mathbb T^{|X\setminus A|}$,
\[
\frac{2\,w_{\min}\,d_{\min}^2}{\pi^2}\,
\mathrm{Lip}_\rho(\chi_\theta)^2
\ \le\
\lambda(\theta)
\ \le\
\frac{\pi^2}{4}\,w_{\max}\,|X\setminus A|\,d_{\max}^2\,
\mathrm{Lip}_\rho(\chi_\theta)^2.
\]
\end{lemma}

\begin{corollary}[Spectral form {\cite[Corollary~2]{fanning2025reproducingkernelhilbertspaces}}]
\label{cor:heat-lip-spectral}
For every $t>0$ and every $f\in\mathcal H_t$,
\[
 \mathrm{Lip}_\rho(f)\ \le\ \frac{\pi}{d_{\min}\sqrt{2\,w_{\min}}}\,
 \|f\|_{\mathcal H_t}\,
 \Bigg(
 \int_{\mathbb T^{|X\setminus A|}}
 \lambda(\theta)\,e^{-t\lambda(\theta)}\,d\mu(\theta)
 \Bigg)^{1/2},
\]
with $w_{\min}$ and $d_{\min}$ as in Lemma~\ref{lem:lambda-vs-L}.
\end{corollary}

\begin{figure}[t]
 \centering
 \includegraphics[width=0.5\linewidth]{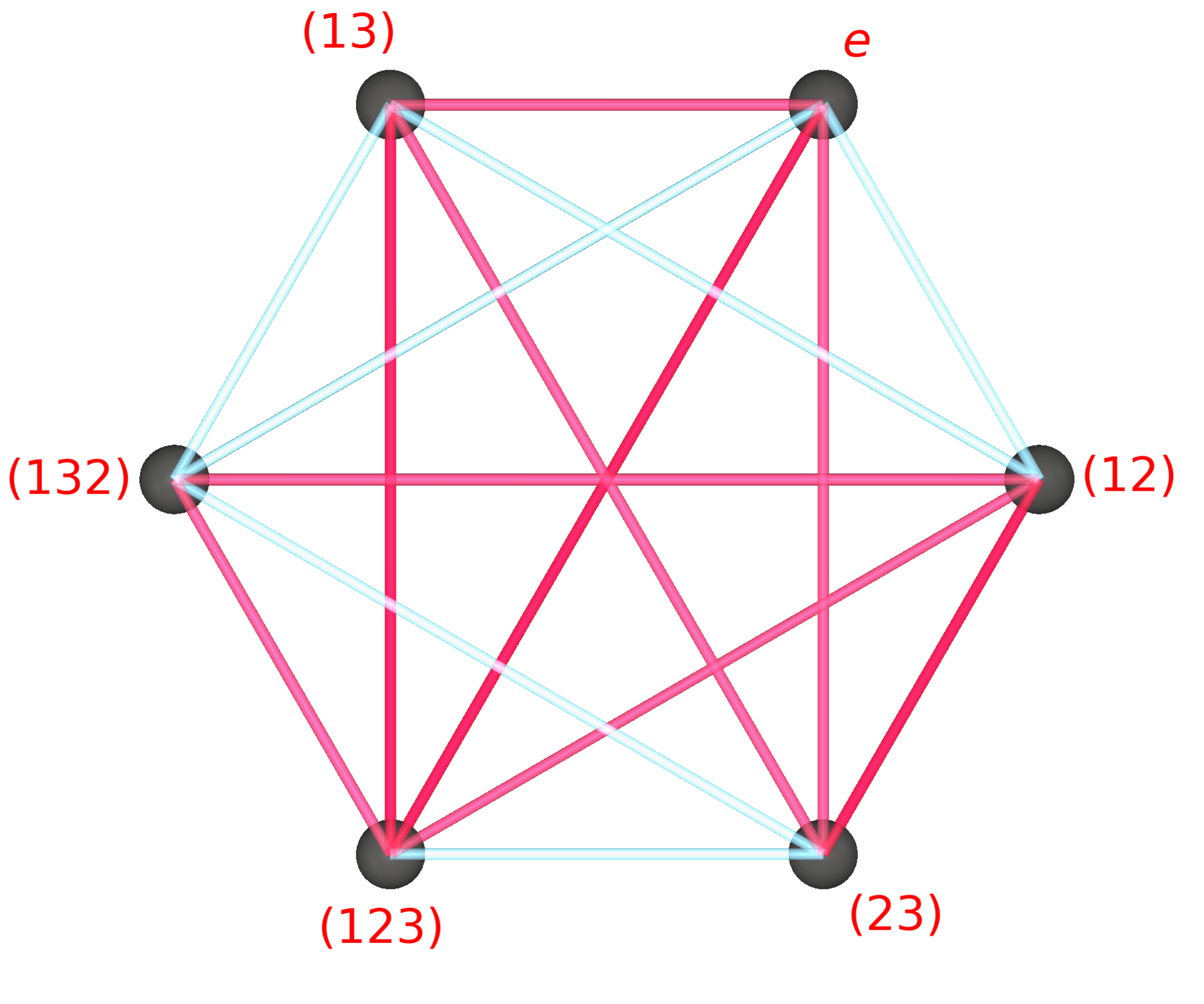}
 \caption{The generators of the virtual persistence diagram group $K(X,A)$ induced by the Cayley graph of $S_3$ with its word metric.}
 \label{fig:rkhs-finite-s3}
\end{figure}

Figure~\ref{fig:rkhs-finite-s3} visualizes a subclass of translation--invariant kernels in the reproducing kernel Hilbert space $\mathcal H_t$ associated with the heat semigroup on the virtual persistence diagram group $K(X,A),$ pairwise--distance kernels whose values on virtual persistence diagrams $g,h\in K(X,A)$ are determined by a fixed $\psi:[0,\infty)\to\mathbb R$ by $\psi(r)=e^{-\alpha r^2}$ through
\[
k(g,h)
=
\sum_{x,y\in X/A} m_g(x)\,m_h(y)\,\psi\!\left(\overline d_1(x,y)\right),
\]
where $m_g(x)\in\mathbb Z$ denotes the signed multiplicity of the generator $x\in X/A$ in $g$. For each pair of generators $x,y\in X/A$, the edge connecting them is colored according to the value $\psi(\overline d_1(x,y))=k(\delta_x,\delta_y)$, with lighter colors corresponding to smaller word--metric distances and darker colors to larger distances. An arbitrary virtual persistence diagram may therefore be visualized by placing its signed multiplicities on the vertices of the graph, and kernel evaluations within this subclass are obtained by summing products of vertex multiplicities weighted by the edge values encoded by the coloring. In this way, the entire family of pairwise--distance kernels with fixed profile $\psi$ is represented by the edge--colored Cayley graph.

\subsubsection{Discrete virtual persistence diagrams}

Assume that the pointed metric space $(X/A,\overline d_1,[A])$ is uniformly discrete. Then $(K(X,A),\rho)$ is a discrete locally compact abelian group \cite{fanning2026reproducingkernelhilbertspaces}. In particular, $K(X,A)\cong\bigoplus_{x\in X\setminus A}\mathbb Z$, and its Pontryagin dual is $\widehat K\cong\prod_{x\in X\setminus A}\mathbb T$.

\subsection{Markov semigroups}\label{subsec:background-markov}

Let $H$ be a countable discrete abelian group. A family $\{P(t),t\ge0\}$
of bounded linear operators on $\ell^\infty(H)$ is called a
\emph{Markov semigroup}~\cite[Def.~3.4]{liggett2010ctmp} if:

\begin{enumerate}
 \item[(a)] $P(0)f = f$ for all $f\in\ell^\infty(H)$.
 \item[(b)] For every $f\in\ell^\infty(H)$,
 $\displaystyle\lim_{t\downarrow0}\|P(t)f - f\|_\infty = 0$.
 \item[(c)] $P(s+t)f = P(s)P(t)f$ for every $f\in\ell^\infty(H)$
 and all $s,t\ge0$.
 \item[(d)] $P(t)f \ge 0$ for every nonnegative
 $f\in\ell^\infty(H)$ and every $t\ge0$.
 \item[(e)] $P(t)\mathbf 1 = \mathbf 1$ for each $t>0$.
\end{enumerate}

The semigroup is called \emph{symmetric} if \(p_t(h)=p_t(-h), \) for all $h \in H$ and $t \geq 0.$

For $\alpha\in H$, define the translation operator by $(\tau_\alpha f)(\beta)=f(\beta-\alpha)$ for $\beta\in H$. A Markov semigroup $(P(t))_{t\ge0}$ on $\ell^\infty(H)$ is called \emph{translation invariant} if $P(t)\tau_\alpha=\tau_\alpha P(t)$ for all $t\ge0$ and $\alpha\in H$.

If $(P(t))_{t\ge0}$ is a symmetric, strongly continuous Markov semigroup on $\ell^\infty(H)$, then it extends to a strongly continuous contraction semigroup on $\ell^2(H)$. Consequently there exists a densely defined, nonnegative, self-adjoint operator $L$ on $\ell^2(H)$ such that $P(t)=e^{-tL}$ for $t\ge0$ on $\ell^2(H)$. The operator $-L$ is called the (infinitesimal) generator of the semigroup.

\section{Infinite-dimensional heat kernel RKHS on discrete VPD groups}
\label{sec:heat-rkhs}

In the case where \(X\setminus A\) is finite, the virtual persistence diagram group \(K(X,A)\) is a discrete locally compact abelian group with Pontryagin dual canonically identified as a torus \(\widehat K\cong\mathbb T^{|X\setminus A|}\) \cite{fanning2025reproducingkernelhilbertspaces}. When \(X\setminus A\) is infinite, the group \(K(X,A)\) need not be locally compact, and the Pontryagin duality used in the finite case is no longer applicable \cite{fanning2026reproducingkernelhilbertspaces}. The virtual persistence diagram group \(K(X,A)\) is locally compact abelian if and only if the pointed metric space \((X/A,\overline d_1,[A])\) is uniformly discrete. Motivated by this equivalence, we now fix uniform discreteness as a standing assumption for the remainder of the section.

\begin{definition}\label{def:standing-assumption}
The \emph{standing assumption} is:
\begin{enumerate}
\item[(H)]
The pointed metric space \((X/A,\overline d_1,[A])\) is uniformly discrete.
\end{enumerate}
\end{definition}

Under assumption~\emph{(H)}, the virtual persistence diagram group \((K(X,A),\rho)\) is a locally compact abelian group.

We construct a symmetric translation-invariant jump kernel on \(K(X,A)\) supported on elementary pair-jumps. Under the identification fixed above, \(K(X,A)\cong \bigoplus_{x\in X\setminus A}\mathbb Z e_x\), we write \(e_x\) for the basis element corresponding to \(x\in X\setminus A\). The finite case assigns a weight to each ordered pair \((x,y)\) with \(x\neq y\) and associates to it the increment \(e_x-e_y\); extending this to infinite \(X\setminus A\) leads to assigning weights to all such pairs simultaneously. In that case, the total outgoing jump rate is 
\[
\sum_{\substack{(x,y)\in (X\setminus A)\times (X\setminus A) \\ x\neq y}} \psi\bigl(\overline d_1(x,y)\bigr),
\]
which may be infinite. In particular, the sum of the weights of all admissible pair-jumps need not be finite, so the naive infinite extension does not define a finite symmetric kernel. To obtain a well-defined object with finite total mass, we impose the summability condition \eqref{eq:psi-summability}.

Let \(\psi:[0,\infty)\to[0,\infty)\) satisfy \(\psi(0)=0\). Define $E = \{(x,y)\in (X\setminus A)\times (X\setminus A) : x\neq y\}$.
We assume
\begin{equation}\label{eq:psi-summability}
\sum_{(x,y)\in E}\psi\bigl(\overline d_1(x,y)\bigr)<\infty,
\end{equation}
where, since \(E\) may be uncountable, the sum is understood as the supremum of the sums over finite subsets of \(E\).

\begin{lemma}\label{lem:countable-support}
The set $\{(x,y)\in E:\psi\bigl(\overline d_1(x,y)\bigr)>0\}$ is countable.
\end{lemma}
\begin{proof}
For each \(m\in\mathbb N\), set
\[
A_m=\left\{(x,y)\in E:\psi\bigl(\overline d_1(x,y)\bigr)>\frac{1}{m}\right\}.
\]
If some \(A_m\) were infinite, then for each \(N\in\mathbb N\) there would exist a subset \(J\subset A_m\) with \(|J|=N\), and hence
\[
\sum_{(x,y)\in J}\psi\bigl(\overline d_1(x,y)\bigr)>\frac{N}{m},
\]
which contradicts \eqref{eq:psi-summability} once \(N\) is sufficiently large. Thus each \(A_m\) is finite. Since
\[
\{(x,y)\in E:\psi\bigl(\overline d_1(x,y)\bigr)>0\}
=\bigcup_{m=1}^\infty A_m,
\]
the set is countable.
\end{proof}

We define a function $\nu:K(X,A)\setminus\{0\}\to[0,\infty)$, supported on pair-jumps, by
\begin{equation}\label{eq:nu}
\nu(\kappa)
=
\begin{cases}
\dfrac{1}{2}\psi\bigl(\overline d_1(x,y)\bigr), & \kappa=e_x-e_y,\ x,y\in X\setminus A,\ x\neq y,\\[0.5em]
0, & \text{otherwise}.
\end{cases}
\end{equation}
Since $\{e_x : x\in X\setminus A\}$ is a free generating set, the map $(x,y)\mapsto e_x-e_y$ is injective on ordered pairs with $x\neq y$. Therefore
\[
\sum_{\kappa\in K(X,A)\setminus\{0\}}\nu(\kappa)
=
\frac{1}{2}
\sum_{(x,y)\in E}\psi\bigl(\overline d_1(x,y)\bigr)
<\infty.
\]
Also, $\nu(-\kappa)=\nu(\kappa)$ for all $\kappa$, and $\nu\in\ell^1(K(X,A)\setminus\{0\})$.

The Fourier symbol $\lambda^\psi:\widehat{K(X,A)}\to[0,\infty)$ is defined by
\begin{equation}\label{eq:global-symbol}
\lambda^\psi(\theta)
=
\sum_{\kappa\in K(X,A)\setminus\{0\}}
\nu(\kappa)\bigl(1-\Re \theta(\kappa)\bigr).
\end{equation}
Since \(0\le 1-\Re\theta(\kappa)\le 2\) for all \(\theta\in\widehat{K(X,A)}\) and all \(\kappa\), the series in \eqref{eq:global-symbol} converges absolutely and uniformly on \(\widehat{K(X,A)}\). In particular, $\lambda^\psi$ is well defined, continuous as a uniform limit of continuous functions, and nonnegative.

To approximate the global symbol by finite-rank truncations, fix a finite set \(F\subset X\setminus A\) and set \(X_F=A\cup F\). Then
\[
K(X_F,A)\cong \bigoplus_{x\in F}\mathbb Z e_x.
\]

We define the truncated symbol $\lambda_F^{D,\psi}:\widehat{K(X_F,A)}\to[0,\infty)$ by
\begin{equation}\label{eq:finite-symbol}
\begin{aligned}
\lambda_F^{D,\psi}(\chi)
&=
\frac{1}{2}
\sum_{\substack{x,y\in F\\ x\neq y}}
\psi\bigl(\overline d_1(x,y)\bigr)
\bigl(1-\Re \chi(e_x-e_y)\bigr) \\
&+
\sum_{x\in F}
\left(
\sum_{y\in (X\setminus A)\setminus F}
\psi\bigl(\overline d_1(x,y)\bigr)
\right)
\bigl(1-\Re \chi(e_x)\bigr).
\end{aligned}
\end{equation}
For each \(x\in F\), the inner sum is finite by \eqref{eq:psi-summability}, so \eqref{eq:finite-symbol} is well defined. The first term corresponds to pairs in $F$, and the second to pairs with exactly one endpoint in $F$.

\begin{lemma}\label{lem:projection}
For every finite \(F\subset X\setminus A\) and every \(\chi\in \widehat{K(X_F,A)}\),
\begin{equation}\label{eq:projection}
\lambda^\psi(\chi\circ \pi_F)=\lambda_F^{D,\psi}(\chi).
\end{equation}
\end{lemma}

\begin{proof}
Using \eqref{eq:global-symbol} and absolute convergence, we write
\[
\lambda^\psi(\chi\circ \pi_F)
=
\frac{1}{2}
\sum_{\substack{x,y\in X\setminus A\\ x\neq y}}
\psi\bigl(\overline d_1(x,y)\bigr)
\bigl(1-\Re (\chi\circ \pi_F)(e_x-e_y)\bigr).
\]
If $x,y\in F$, then $\pi_F(e_x-e_y)=e_x-e_y$, and these terms give the first sum in \eqref{eq:finite-symbol}. If $x\in F$ and $y\notin F$, then $\pi_F(e_x-e_y)=e_x$. If $x\notin F$ and $y\in F$, then $\pi_F(e_x-e_y)=-e_y$, and $1-\Re\chi(-e_y)=1-\Re\chi(e_y)$.
Hence, for each fixed \(x\in F\), the contributions from the ordered pairs \((x,y)\) and \((y,x)\) with \(y\in (X\setminus A)\setminus F\) combine to
\[
\frac{1}{2}\psi\bigl(\overline d_1(x,y)\bigr)\bigl(1-\Re\chi(e_x)\bigr)
+
\frac{1}{2}\psi\bigl(\overline d_1(y,x)\bigr)\bigl(1-\Re\chi(e_x)\bigr)
=
\psi\bigl(\overline d_1(x,y)\bigr)\bigl(1-\Re\chi(e_x)\bigr),
\]
where we used the symmetry \(\overline d_1(y,x)=\overline d_1(x,y)\). Summing over \(y\in (X\setminus A)\setminus F\) gives exactly the boundary term in \eqref{eq:finite-symbol}. If \(x,y\notin F\), then \(\pi_F(e_x-e_y)=0\), so the contribution vanishes. This proves \eqref{eq:projection}. Since \eqref{eq:finite-symbol} is a finite nonnegative linear combination of functions of the form $1-\Re\chi(\eta)$, it is continuous.
\end{proof}

We define the support of $\nu$ by
$\operatorname{supp}(\nu) = \{\kappa \in K(X,A)\setminus\{0\} : \nu(\kappa)>0\}$.
By Lemma~\ref{lem:countable-support} and the definition \eqref{eq:nu}, the set $\operatorname{supp}(\nu)$ is countable.

Thus the finite-rank symbol is obtained by pulling back the global symbol along \(\pi_F\).

\begin{theorem}\label{thm:global}
There exists a unique convolution semigroup of probability measures \((p_t^\psi)_{t\ge 0}\) on \(K(X,A)\) such that, for every \(\theta\in \widehat{K(X,A)}\),
\begin{equation}\label{eq:fourier}
\widehat{p_t^\psi}(\theta)=e^{-t\lambda^\psi(\theta)}.
\end{equation}
\end{theorem}

\begin{proof}
The series \eqref{eq:global-symbol} converges absolutely and uniformly, so \(\lambda^\psi\) is bounded and continuous as a uniform limit of bounded continuous functions.

Let \(\theta_1,\dots,\theta_n\in \widehat{K(X,A)}\) and \(c_1,\dots,c_n\in\mathbb C\) satisfy \(\sum_{j=1}^n c_j=0\). Enumerate the countable support of \(\nu\) as \(\{\kappa_1,\kappa_2,\dots\}\), and define
\[
\lambda_m^\psi(\theta)
=
\sum_{k=1}^m
\nu(\kappa_k)\bigl(1-\Re \theta(\kappa_k)\bigr).
\]
Each function \(\theta\mapsto 1-\Re\theta(\kappa_k)\) is continuous and negative definite, and finite nonnegative linear combinations preserve negative definiteness; hence \(\lambda_m^\psi\) is continuous and negative definite, and
\[
\sum_{i,j=1}^n
\lambda_m^\psi(\theta_i^{-1}\theta_j)c_i\overline{c_j}
\le 0.
\]
Since \(\lambda_m^\psi(\theta)\to \lambda^\psi(\theta)\) for every \(\theta\), and the defining quadratic form is preserved under pointwise limits, taking the limit gives
\[
\sum_{i,j=1}^n
\lambda^\psi(\theta_i^{-1}\theta_j)c_i\overline{c_j}
\le 0,
\]
so \(\lambda^\psi\) is negative definite.

By Schoenberg's theorem, \(e^{-t\lambda^\psi}\) is positive definite for each \(t\ge 0\). Since it is continuous and equals \(1\) at the trivial character, and \(K(X,A)\) is locally compact abelian under the standing assumption, Bochner's theorem gives a unique probability measure $p_t^\psi$ satisfying \eqref{eq:fourier}. The identity $e^{-(s+t)\lambda^\psi}=e^{-s\lambda^\psi}e^{-t\lambda^\psi}$ implies \(p_{s+t}^\psi=p_s^\psi*p_t^\psi\), so \((p_t^\psi)_{t\ge0}\) is a convolution semigroup.

For \(t\ge0\), define \(P_t f = p_t^\psi * f\). Then \(P_t\) is a symmetric, translation-invariant, positivity-preserving contraction on \(\ell^p(K(X,A))\) for \(1\le p\le\infty\). It is strongly continuous on \(\ell^2(K(X,A))\): by Plancherel,
\[
\|P_t f-f\|_{\ell^2(K(X,A))}^2
=
\int_{\widehat{K(X,A)}}
|e^{-t\lambda^\psi(\theta)}-1|^2\,|\widehat f(\theta)|^2\,d\mu_{\widehat{K(X,A)}}(\theta),
\]
and dominated convergence applies since \(\lambda^\psi\) is bounded.

Lemma~\ref{lem:projection} implies that for every finite \(F\subset X\setminus A\),
\begin{align*}
\widehat{(\pi_F)_*p_t^\psi}(\chi)
&=
\widehat{p_t^\psi}(\chi\circ \pi_F) \\
&=
e^{-t\lambda^\psi(\chi\circ \pi_F)} \\
&=
e^{-t\lambda_F^{D,\psi}(\chi)},
\end{align*}
so the convolution semigroup on \(K(X_F,A)\) generated by \(\lambda_F^{D,\psi}\) is the pushforward of \((p_t^\psi)_{t\ge0}\) under \(\pi_F\).

Define an operator \(L^\psi\) on finitely supported \(f:K(X,A)\to\mathbb C\) by
\[
(L^\psi f)(\gamma)
=
\sum_{\kappa\in K(X,A)\setminus\{0\}}
\nu(\kappa)\bigl(f(\gamma)-f(\gamma+\kappa)\bigr).
\]
Since \(\nu\in \ell^1(K(X,A)\setminus\{0\})\), the sum converges absolutely and
\[
\|L^\psi f\|_{\ell^1}
\le
2\Bigl(\sum_{\kappa\neq 0}\nu(\kappa)\Bigr)\|f\|_{\ell^1},
\]
so \(L^\psi\) extends uniquely to a bounded operator on \(\ell^1(K(X,A))\).

For finitely supported \(f\), the sum over \(\kappa\) is absolutely summable and the sum over \(\gamma\) is finite, so Fubini's theorem justifies interchange of the sum and Fourier transform, giving
\[
\widehat{L^\psi f}(\theta)
=
\sum_{\kappa\neq 0}\nu(\kappa)\bigl(1-\theta(\kappa)\bigr)\widehat f(\theta).
\]
Using the symmetry \(\nu(\kappa)=\nu(-\kappa)\), we write
\begin{align*}
\sum_{\kappa\neq 0}\nu(\kappa)\bigl(1-\theta(\kappa)\bigr)
&=
\frac{1}{2}
\sum_{\kappa\neq 0}
\Bigl[
\nu(\kappa)\bigl(1-\theta(\kappa)\bigr)
+
\nu(-\kappa)\bigl(1-\theta(-\kappa)\bigr)
\Bigr] \\
&=
\frac{1}{2}
\sum_{\kappa\neq 0}
\nu(\kappa)\Bigl[(1-\theta(\kappa))+(1-\overline{\theta(\kappa)})\Bigr] \\
&=
\sum_{\kappa\neq 0}\nu(\kappa)\bigl(1-\Re\theta(\kappa)\bigr)
=
\lambda^\psi(\theta).
\end{align*}
Thus $\widehat{L^\psi f}(\theta)=\lambda^\psi(\theta)\widehat f(\theta)$.

Since \(L^\psi\) is bounded on \(\ell^1(K(X,A))\), the operator exponential
\[
e^{-tL^\psi}
=
\sum_{n=0}^\infty \frac{(-t)^n}{n!}(L^\psi)^n
\]
defines a uniformly continuous semigroup on $\ell^1(K(X,A))$. For finitely supported $f$, $\widehat{e^{-tL^\psi}f}(\theta)=e^{-t\lambda^\psi(\theta)}\widehat f(\theta)$.
On the other hand,
\begin{align*}
\widehat{P_t f}(\theta)
&=
\widehat{p_t^\psi}(\theta)\widehat f(\theta) \\
&=
e^{-t\lambda^\psi(\theta)}\widehat f(\theta).
\end{align*}
Hence $\widehat{P_t f}=\widehat{e^{-tL^\psi}f}$, and by injectivity of the Fourier transform on $\ell^1(K(X,A))$ for discrete abelian groups, we conclude that $P_t f=e^{-tL^\psi}f$. for all finitely supported \(f\). By density of finitely supported functions in \(\ell^1(K(X,A))\) and continuity of both semigroups, this identity extends to all \(f\in \ell^1(K(X,A))\).

It follows that \((P_t)_{t\ge0}\) is uniformly continuous on \(\ell^1(K(X,A))\), and in particular
\[
\|p_t^\psi-p_s^\psi\|_{\ell^1}
=
\|P_t\delta_0-P_s\delta_0\|_{\ell^1}
\to 0
\quad\text{as }t\to s.
\]

Finally, substituting the definition of \(\nu\) gives
$$(L^\psi f)(\gamma)
=
\frac{1}{2} \sum_{\substack{x,y\in X\setminus A\\ x\neq y}} \psi\bigl(\overline d_1(x,y)\bigr) \bigl(f(\gamma)-f(\gamma+e_x-e_y)\bigr),$$
which completes the proof.
\end{proof}

The global symbol pulls back to the finite-rank symbol along \(\pi_F\), and the finite-rank semigroup is the pushforward of the global semigroup.

\begin{lemma}\label{lem:effective-support}
Let \((p_t^\psi)_{t\ge0}\) be the convolution semigroup on \(K(X,A)\) constructed above. Then the subgroup \(\big\langle \bigcup_{q\in\mathbb Q_{\ge0}} \operatorname{supp}(p_q^\psi) \big\rangle\) is countable, and
\[
\operatorname{supp}(p_t^\psi)\subset \bigcup_{q\in\mathbb Q_{\ge0}} \operatorname{supp}(p_q^\psi)
\]
for all \(t\ge0\).
\end{lemma}

\begin{proof}
Each \(p_q^\psi\) is a probability mass function on the discrete group \(K(X,A)\), so \(\operatorname{supp}(p_q^\psi)\) is countable. Since \(\mathbb Q_{\ge0}\) is countable, the union \(\bigcup_{q\in\mathbb Q_{\ge0}} \operatorname{supp}(p_q^\psi)\) is countable, and hence so is the subgroup it generates.

Fix \(\gamma\in K(X,A)\) and define \(f_\gamma(t)=p_t^\psi(\gamma)\). The map \(t\mapsto p_t^\psi\) is continuous as a function into \(\ell^1\bigl(K(X,A)\bigr)\), and evaluation at \(\gamma\) is a bounded linear functional, so \(f_\gamma\) is continuous.

If \(p_t^\psi(\gamma)>0\) for some \(t\ge0\), then by continuity there exists \(\varepsilon>0\) such that \(p_s^\psi(\gamma)>0\) for all \(s\in(t-\varepsilon,t+\varepsilon)\cap[0,\infty)\). Since \(\mathbb Q_{\ge0}\) is dense in \([0,\infty)\), there exists \(q\in\mathbb Q_{\ge0}\) in this interval, and hence \(\gamma\in \operatorname{supp}(p_q^\psi)\). Therefore \(\gamma\) lies in \(\bigcup_{q\in\mathbb Q_{\ge0}} \operatorname{supp}(p_q^\psi)\).
\end{proof}

Set
\[
H = \Big\langle \bigcup_{q\in\mathbb Q_{\ge0}} \operatorname{supp}(p_q^\psi) \Big\rangle \subset K(X,A).
\]
By Lemma~\ref{lem:effective-support}, the group \(H\) is countable and \(\operatorname{supp}(p_t^\psi)\subset H\) for every \(t\ge0\). Thus \((p_t^\psi)_{t\ge0}\) may be regarded as a symmetric convolution semigroup of probability measures on \(H\), and from this point onward all analysis is carried out on \(H\).

For $\theta\in\widehat H$, define
\begin{equation*}
\widehat p_t^\psi(\theta)=\sum_{h\in H} p_t^\psi(h)\,\overline{\theta(h)}
\end{equation*}
for $t\ge0$. Since $p_t^\psi$ is supported on $H$, and every character on $H$ extends to a character on $K(X,A)$ (because $\mathbb T$ is divisible), any extension $\tilde\theta\in\widehat{K(X,A)}$ of $\theta$ satisfies $\widehat p_t^\psi(\theta)=\widehat{p_t^\psi}(\tilde\theta)=e^{-t\lambda^\psi(\tilde\theta)}$. Hence there exists a function $\lambda_H:\widehat H\to[0,\infty)$ such that $\widehat p_t^\psi(\theta)=e^{-t\lambda_H(\theta)}$ for $\theta\in\widehat H$ and $t\ge0$.

\begin{theorem}
\label{thm:lk-H}
There exists, for all $\theta\in\widehat H$, a unique symmetric $\ell^1(H\setminus\{0\})\ni \nu\colon H\setminus\{0\}\to[0,\infty)$ such that
\begin{equation}
\lambda_H(\theta)
=
\sum_{\kappa\in H\setminus\{0\}}
\bigl(1-\Re\,\theta(\kappa)\bigr)\,\nu(\kappa)
\label{eq:lk-series}
\end{equation}
\end{theorem}

\begin{proof}
For $t\ge0$, define $P_t$ on $\ell^1(H)$ by
\begin{equation*}
(P_t f)(\gamma)
=
\sum_{h\in H} p_t^\psi(h)\,f(\gamma-h)
\end{equation*}
for $f\in\ell^1(H)$ and $\gamma\in H$. Then for $f\in\ell^1(H)$,
\[
\|P_t f - P_s f\|_{\ell^1(H)}
\le
\|p_t^\psi - p_s^\psi\|_{\ell^1(H)} \|f\|_{\ell^1(H)},
\]
so $(P_t)_{t\ge 0}$ is a uniformly continuous convolution semigroup on $\ell^1(H)$.

Let $B$ denote its generator,
\begin{equation*}
Bf
=
\lim_{t\downarrow 0} \frac{P_t f - f}{t}
\end{equation*}
for $f\in\ell^1(H)$, which exists in $\ell^1(H)$ and defines a bounded linear operator.

Each $P_t$ is translation invariant, hence $B$ commutes with translations. Let $\delta_0$ denote the unit mass at $0$, and set $\eta = B\delta_0 \in \ell^1(H)$. If $f$ has finite support, then
\[
f = \sum_{\gamma\in H} f(\gamma)\,\tau_\gamma \delta_0,
\]
where $(\tau_\gamma g)(x)=g(x-\gamma)$. Using translation invariance of $B$, we obtain
\[
Bf = \sum_{\gamma\in H} f(\gamma)\,\tau_\gamma (B\delta_0) = \eta * f.
\]
By density of finitely supported functions in $\ell^1(H)$ and boundedness of $B$, this identity extends to all $f\in \ell^1(H)$. Thus $\eta\in\ell^1(H)$ is a finite signed measure on $H$.

Since $P_t$ preserves total mass, we have $\sum_{\kappa\in H} \eta(\kappa)=0$. For $\kappa\ne 0$,
\[
\eta(\kappa)
=
(B\delta_0)(\kappa)
=
\lim_{t\downarrow 0} \frac{p_t^\psi(\kappa)}{t}
\ge 0,
\]
since $p_t^\psi(\kappa)\ge0$. Define $\ell^1(H\setminus\{0\})\ni \nu:H\setminus\{0\}\to[0,\infty)$ by $\nu(\kappa)=\eta(-\kappa)$ for $\kappa\ne0$.
Then $\nu\in \ell^1(H\setminus\{0\})$, and for $\gamma\in H$,
\[
(Bf)(\gamma)
=
\sum_{\kappa\ne 0} \nu(\kappa)\bigl(f(\gamma+\kappa)-f(\gamma)\bigr).
\]
Let $f$ have finite support. Then
\begin{align*}
\widehat{Bf}(\theta)
&=
\sum_{\gamma\in H} (Bf)(\gamma)\,\overline{\theta(\gamma)} \\
&=
\sum_{\kappa\ne 0} \nu(\kappa)
\sum_{\gamma\in H} \bigl(f(\gamma+\kappa)-f(\gamma)\bigr)\overline{\theta(\gamma)}.
\end{align*}
A change of variables in the first term gives
\[
\sum_{\gamma\in H} f(\gamma+\kappa)\overline{\theta(\gamma)}
=
\theta(\kappa)\widehat f(\theta),
\]
and therefore
\[
\widehat{Bf}(\theta)
=
-\Bigl(\sum_{\kappa\ne 0} \nu(\kappa)\bigl(1-\theta(\kappa)\bigr)\Bigr)\widehat f(\theta).
\]
On the other hand, $\widehat{P_t f}(\theta)=e^{-t\lambda_H(\theta)}\widehat f(\theta)$, and since $(P_t)$ is uniformly continuous,
\[
\widehat{Bf}(\theta)
=
\lim_{t\downarrow 0} \frac{\widehat{P_t f}(\theta)-\widehat f(\theta)}{t}
=
-\lambda_H(\theta)\widehat f(\theta).
\]
Hence
\[
\lambda_H(\theta)
=
\sum_{\kappa\ne 0} \nu(\kappa)\bigl(1-\theta(\kappa)\bigr).
\]

Since $p_t^\psi$ is symmetric, we have $\widehat p_t^\psi(\theta)\in\mathbb R$ for all $\theta\in\widehat H$. Since $\widehat p_t^\psi(\theta)=e^{-t\lambda_H(\theta)}$, it follows that $\lambda_H(\theta)$ is real-valued for all $\theta$. Thus
\begin{align*}
\sum_{\kappa\ne 0} \nu(\kappa)\theta(\kappa)
&=
\overline{\sum_{\kappa\ne 0} \nu(\kappa)\theta(\kappa)} \\
&=
\sum_{\kappa\ne 0} \nu(\kappa)\overline{\theta(\kappa)} \\
&=
\sum_{\kappa\ne 0} \nu(\kappa)\theta(-\kappa).
\end{align*}
Reindexing the last sum gives
\[
\sum_{\kappa\ne 0} \nu(\kappa)\theta(\kappa)
=
\sum_{\kappa\ne 0} \nu(-\kappa)\theta(\kappa).
\]
By injectivity of the Fourier transform on $\ell^1(H)$, it follows that $\nu(\kappa)=\nu(-\kappa)$ for all $\kappa$. Therefore
\[
\sum_{\kappa\ne 0} \nu(\kappa)\theta(\kappa)
=
\sum_{\kappa\ne 0} \nu(\kappa)\Re\,\theta(\kappa),
\]
and substituting into the expression for $\lambda_H$ gives \eqref{eq:lk-series}.

We now prove uniqueness. Suppose $\nu_1$ and $\nu_2$ both satisfy \eqref{eq:lk-series}, and set $\sigma=\nu_1-\nu_2$. Then $\sigma$ is symmetric and for all $\theta\in \widehat H$,
\[
\sum_{\kappa\ne 0} (1-\Re\,\theta(\kappa))\sigma(\kappa)=0.
\]
Define a finite signed measure $\mu$ on $H$ by
\begin{equation*}
\mu(0)=-\sum_{\kappa\ne0}\sigma(\kappa), \qquad
\mu(\kappa)=\sigma(\kappa)\ \text{for }\kappa\ne0
\end{equation*}
Then
\[
\widehat\mu(\theta)
=
-\sum_{\kappa\ne 0}\sigma(\kappa)
+
\sum_{\kappa\ne 0}\sigma(\kappa)\overline{\theta(\kappa)}.
\]
Since \(\sigma\) is symmetric and the series converges absolutely, we may pair \(\kappa\) and \(-\kappa\) to obtain
\[
\sum_{\kappa\ne 0}\sigma(\kappa)\overline{\theta(\kappa)}
=
\sum_{\kappa\ne 0}\sigma(\kappa)\Re\,\theta(\kappa).
\]
Therefore
\[
\widehat\mu(\theta)
=
-\sum_{\kappa\ne 0}\sigma(\kappa)\bigl(1-\Re\,\theta(\kappa)\bigr)
=
0.
\]
By injectivity of the Fourier transform on finite measures on $H$, we conclude that $\mu=0$, hence $\sigma=0$. Therefore $\nu_1=\nu_2$.
\end{proof}

Fix \(R>0\). For \(\kappa\in H\setminus\{0\}\), we define the metric truncation of the L\'evy measure by \( \nu_R(\kappa) = \nu(\kappa)\,\mathbf 1_{\{\rho(\kappa,0)\le R\}} \). For \(\theta\in\widehat H\), the associated truncated L\'evy exponent is then given by
\[
 \lambda_{H,R}(\theta)
 =
 \sum_{\kappa\in H\setminus\{0\}}
 \bigl(1-\Re\,\theta(\kappa)\bigr)\,\nu_R(\kappa).
\]

\begin{lemma}\label{lem:lambdaHR-monotone}
For every \(\theta\in\widehat H\), we have
\[
\lambda_{H,R}(\theta)\uparrow \lambda_H(\theta)
\quad\text{as }R\to\infty.
\]
\end{lemma}

\begin{proof}
Fix \(\theta\in\widehat H\). By definition,
\[
\lambda_H(\theta)
=
\sum_{\kappa\in H\setminus\{0\}}
\bigl(1-\Re\,\theta(\kappa)\bigr)\,\nu(\kappa),
\]
and
\[
\lambda_{H,R}(\theta)
=
\sum_{\substack{\kappa\in H\setminus\{0\}\\ \rho(\kappa,0)\le R}}
\bigl(1-\Re\,\theta(\kappa)\bigr)\,\nu(\kappa).
\]
Since all summands are nonnegative and \(\nu\in\ell^1(H\setminus\{0\})\), the series defining \(\lambda_H(\theta)\) converges absolutely. The sets \(\{\kappa:\rho(\kappa,0)\le R\}\) increase to \(H\setminus\{0\}\) as \(R\to\infty\), so the result follows by monotone convergence.
\end{proof}

For $q>0$ and a probability measure $\pi$ on $H\setminus\{0\}$, let $(\xi_j)_{j\ge1}$ be i.i.d.\ $H$-valued random variables with law $\pi$, and let $(N_t)_{t\ge0}$ be a Poisson process with rate $q$, independent of $(\xi_j)$. Define
\begin{equation*}
S_t=\sum_{j=1}^{N_t}\xi_j
\end{equation*}
for $t\ge0$. If $q=0$, set $S_t\equiv0$ for all $t\ge0$.

\begin{corollary}\label{cor:finite-activity-compound-poisson}
The following are equivalent, and condition (i) holds under assumption \eqref{eq:psi-summability}:
\begin{itemize}
\item[(i)] \(\displaystyle \sum_{\kappa\in H\setminus\{0\}} \nu(\kappa)<\infty\).
\item[(ii)] Either $p_t^\psi=\delta_0$ for every $t\ge0$, or there exist $q>0$ and a probability measure $\pi$ on $H\setminus\{0\}$ such that, for every $t\ge0$ and $\gamma\in H$, $p_t^\psi(\gamma)=\mathbb P\!\left(S_t=\gamma\right)$.
\end{itemize}
\end{corollary}

\begin{proof}
Assume (i), and set
\[
q=\sum_{\kappa\in H\setminus\{0\}}\nu(\kappa).
\]
If \(q=0\), then \(\nu\equiv 0\), hence \(\lambda_H\equiv 0\). Since \(\lambda_H\equiv 0\), we have \(\widehat p_t^\psi(\theta)=1\) for all \(\theta\), and by injectivity of the Fourier transform on \(H\), it follows that \(p_t^\psi=\delta_0\) for all \(t\ge0\). Thus (ii) holds.

Assume now that $q>0$, and define $\pi(\kappa)=\frac{\nu(\kappa)}{q}$ for $\kappa\in H\setminus\{0\}$. Let $\mu_t$ denote the law of $S_t$. For $\theta\in\widehat H$, conditioning on $N_t$ gives
\[
\widehat\mu_t(\theta)
=
\exp\!\left(
qt\left(
\sum_{\kappa\in H\setminus\{0\}}
\pi(\kappa)\,\overline{\theta(\kappa)}-1
\right)
\right).
\]
Using \(\nu(\kappa)=q\,\pi(\kappa)\), we obtain
\[
\widehat\mu_t(\theta)
=
\exp\!\left(
-t\sum_{\kappa\in H\setminus\{0\}}
\nu(\kappa)\bigl(1-\overline{\theta(\kappa)}\bigr)
\right).
\]
Since \(\nu(\kappa)=\nu(-\kappa)\) and \(\theta(-\kappa)=\overline{\theta(\kappa)}\), we have
\[
\sum_{\kappa\in H\setminus\{0\}}
\nu(\kappa)\,\overline{\theta(\kappa)}
=
\sum_{\kappa\in H\setminus\{0\}}
\nu(\kappa)\,\theta(\kappa),
\]
and hence
\[
\sum_{\kappa\in H\setminus\{0\}}
\nu(\kappa)\,\overline{\theta(\kappa)}
=
\sum_{\kappa\in H\setminus\{0\}}
\nu(\kappa)\,\Re\theta(\kappa).
\]
Therefore
\begin{align*}
\widehat\mu_t(\theta)
&=
\exp\!\left(
-t\sum_{\kappa\in H\setminus\{0\}}
\nu(\kappa)\bigl(1-\Re\theta(\kappa)\bigr)
\right) \\
&=
\widehat p_t^\psi(\theta).
\end{align*}
By injectivity of the Fourier transform on finite measures on \(H\), we conclude that \(\mu_t=p_t^\psi\). Thus (ii) holds.

Conversely, assume (ii). If \(p_t=\delta_0\) for every \(t\ge0\), then \(\lambda_H\equiv 0\). Since \(\lambda_H\equiv 0\), we have \(\widehat p_t^\psi(\theta)=1\) for all \(\theta\), and by injectivity of the Fourier transform on finite measures on \(H\), it follows that \(\nu\equiv 0\), and hence (i) holds.

Assume now that there exist \(q>0\) and a probability measure \(\pi\) on \(H\setminus\{0\}\) such that
\[
p_t^\psi(\gamma)=\mathbb P(S_t=\gamma)
\]
for every \(t\ge0\) and \(\gamma\in H\). Let \(\mu_t\) denote the law of \(S_t\). For \(\theta\in\widehat H\), we have
\[
\widehat p_t^\psi(\theta)
=
\exp\!\left(
qt\left(
\sum_{\kappa\in H\setminus\{0\}}
\pi(\kappa)\,\overline{\theta(\kappa)}-1
\right)
\right).
\]
Define
\[
\widehat\pi(\theta)
=
\sum_{\kappa\in H\setminus\{0\}}
\pi(\kappa)\,\overline{\theta(\kappa)}.
\]
Since $\widehat p_t^\psi(\theta)=e^{-t\lambda_H(\theta)}\in\mathbb R_{>0}$ for all $t>0$, we have $e^{qt(\widehat\pi(\theta)-1)}\in\mathbb R_{>0}$ for all $t>0$. Write $\widehat\pi(\theta)-1=a+ib$. Then $e^{qt(\widehat\pi(\theta)-1)}=e^{qta}e^{iqtb}$. If \(b\neq 0\), then there exists \(t>0\) such that \(qtb\notin \pi\mathbb Z\), and hence \(e^{iqtb}\notin\mathbb R\), a contradiction. Thus \(b=0\), and therefore \(\widehat\pi(\theta)\in\mathbb R\) for all \(\theta\in\widehat H\).

Extend \(\pi\) to a finite measure on \(H\) by setting \(\pi(0)=0\), and define \(\check\pi(\kappa)=\pi(-\kappa)\). Then
\begin{align*}
\widehat{\check\pi}(\theta)
&=
\sum_{\kappa\in H}\check\pi(\kappa)\,\overline{\theta(\kappa)}
=
\sum_{\kappa\in H}\pi(-\kappa)\,\overline{\theta(\kappa)} \\
&=
\sum_{\eta\in H}\pi(\eta)\,\overline{\theta(-\eta)}
=
\sum_{\eta\in H}\pi(\eta)\,\theta(\eta)
=
\overline{\widehat\pi(\theta)}
=
\widehat\pi(\theta).
\end{align*}
By injectivity of the Fourier transform on finite measures on $H$, it follows that $\pi=\check\pi$, that is, $\pi(\kappa)=\pi(-\kappa)$. Hence
\[
\sum_{\kappa\in H\setminus\{0\}}
\pi(\kappa)\,\overline{\theta(\kappa)}
=
\sum_{\kappa\in H\setminus\{0\}}
\pi(\kappa)\,\Re\theta(\kappa),
\]
and therefore
\[
\widehat p_t^\psi(\theta)
=
\exp\!\left(
-t\sum_{\kappa\in H\setminus\{0\}}
\nu(\kappa)\bigl(1-\Re\theta(\kappa)\bigr)
\right).
\]
Comparing this representation of $\lambda_H$ with \eqref{eq:lk-series} and using uniqueness in Theorem~\ref{thm:lk-H}, we obtain $\nu(\kappa)=q\,\pi(\kappa)$ for $\kappa\in H\setminus\{0\}$, and hence
\[
\sum_{\kappa\in H\setminus\{0\}}\nu(\kappa)=q<\infty.
\]
This proves (i).
\end{proof}

Fix $t>0$ and define $k_t:H\times H\to\mathbb R$ by $k_t(x,y)=p_t^\psi(x-y)$. Then \(k_t\) is real and positive definite, and hence determines a reproducing kernel Hilbert space \(\mathcal H_t\) of functions on \(H\).

For each finite subset \(F\subset X\setminus A\), set \(H_F=H\cap K(X_F,A)\), and let \(k_t^F\) denote the restriction of \(k_t\) to \(H_F\times H_F\). The kernel \(k_t^F\) is positive definite and induces a reproducing kernel Hilbert space \(\mathcal H_t^F\) of functions on \(H_F\). For each \(h\in H_F\), the kernel section \(k_t^F(\cdot,h)\) agrees with \(k_t(\cdot,h)\), and this identification extends linearly to an isometric embedding of \(\mathcal H_t^F\) into \(\mathcal H_t\). We identify \(\mathcal H_t^F\) with its image in \(\mathcal H_t\).

For every $h\in H$ there exists a finite $F\subset X\setminus A$ such that $h\in H_F$, hence $k_t(\cdot,h)\in \mathcal H_t^F$. Therefore
$
\overline{\bigcup_{F\subset X\setminus A,\ |F|<\infty} \mathcal H_t^F}
=
\mathcal H_t,
$
and $\mathcal H_t$ is separable.

Fix $t>0$, and let $\mu_{\widehat H}$ denote the normalized Haar measure on $\widehat H$. Set $\mathcal W_t=L^2(\widehat H,e^{t\lambda_H}d\mu_{\widehat H})$. Define a linear map $J_t:\mathcal W_t\to \mathbb C^H$ by
\begin{equation*}
(J_t g)(\gamma)=\int_{\widehat H} g(\theta)\,\overline{\theta(\gamma)}\,d\mu_{\widehat H}(\theta)
\end{equation*}
for $\gamma\in H$. Since $\lambda_H$ is bounded and $\mu_{\widehat H}(\widehat H)<\infty$, every $g\in\mathcal W_t$ belongs to $L^1(\widehat H,\mu_{\widehat H})$ by Cauchy--Schwarz. Hence $J_t g$ is well defined.

\begin{lemma}\label{lem:spectral-Ht}
The map \(J_t\) is an isometric isomorphism from \(\mathcal W_t\) onto \(\mathcal H_t\). In particular, a function \(f:H\to\mathbb C\) belongs to \(\mathcal H_t\) if and only if there exists a unique \(g\in \mathcal W_t\) such that
\begin{equation*}
f(\gamma)=\int_{\widehat H} g(\theta)\,\overline{\theta(\gamma)}\,d\mu_{\widehat H}(\theta)
\end{equation*}
for all $\gamma\in H$, and in that case
\[
\|f\|_{\mathcal H_t}^2
=
\int_{\widehat H}
|g(\theta)|^2\,e^{t\lambda_H(\theta)}\,d\mu_{\widehat H}(\theta).
\]
\end{lemma}

\begin{proof}
Since $\lambda_H\ge0$ is bounded, we have $1\le e^{t\lambda_H(\theta)}\le e^{t\|\lambda_H\|_\infty}$ for all $\theta\in\widehat H$. Thus $\mathcal W_t$ coincides with $L^2(\widehat H,\mu_{\widehat H})$ as a set, and the two norms are equivalent.

For each $\gamma\in H$, define $K_{t,\gamma}(\theta)=e^{-t\lambda_H(\theta)}\,\theta(\gamma)$. Then \(K_{t,\gamma}\in \mathcal W_t\), since \(|K_{t,\gamma}(\theta)|^2=e^{-2t\lambda_H(\theta)}\) and hence \(|K_{t,\gamma}(\theta)|^2 e^{t\lambda_H(\theta)}=e^{-t\lambda_H(\theta)}\) is bounded. For \(g\in\mathcal W_t\),
\begin{align*}
(J_t g)(\gamma)
&=
\int_{\widehat H} g(\theta)\,\overline{\theta(\gamma)}\,d\mu_{\widehat H}(\theta) \\
&=
\int_{\widehat H}
g(\theta)\,\overline{K_{t,\gamma}(\theta)}\,e^{t\lambda_H(\theta)}\,d\mu_{\widehat H}(\theta) \\
&=
\langle g,K_{t,\gamma}\rangle_{\mathcal W_t}.
\end{align*}

If \(J_t g=0\), then for all \(\gamma\in H\),
\[
\int_{\widehat H} g(\theta)\,\overline{\theta(\gamma)}\,d\mu_{\widehat H}(\theta)=0.
\]
Since the trigonometric polynomials form a dense subspace of \(L^2(\widehat H,\mu_{\widehat H})\) for the compact abelian group \(\widehat H\), it follows that \(g=0\) in \(L^2(\widehat H,\mu_{\widehat H})\), and hence in \(\mathcal W_t\). Thus \(J_t\) is injective.

We equip $J_t(\mathcal W_t)$ with the inner product transported from $\mathcal W_t$, namely $\langle J_t g_1, J_t g_2\rangle = \langle g_1,g_2\rangle_{\mathcal W_t}$. With this inner product, for \(g\in\mathcal W_t\) and \(\gamma\in H\),
\[
\langle J_t g, J_t K_{t,\gamma}\rangle
=
\langle g,K_{t,\gamma}\rangle_{\mathcal W_t}
=
(J_t g)(\gamma),
\]
so evaluation is represented by \(J_t K_{t,\gamma}\).

For \(\eta\in H\), we compute
\begin{align*}
(J_t K_{t,\eta})(\gamma)
&=
\int_{\widehat H}
e^{-t\lambda_H(\theta)}\,\theta(\eta)\,\overline{\theta(\gamma)}
\,d\mu_{\widehat H}(\theta) \\
&=
\int_{\widehat H}
e^{-t\lambda_H(\theta)}\,\theta(\eta-\gamma)
\,d\mu_{\widehat H}(\theta).
\end{align*}
By the Fourier inversion formula corresponding to $\widehat{p_t^\psi}(\theta)=e^{-t\lambda_H(\theta)}$, this equals $p_t^\psi(\eta-\gamma)$. Since $p_t^\psi$ is symmetric, we have $p_t^\psi(\eta-\gamma)=p_t^\psi(\gamma-\eta)=k_t(\gamma,\eta)$. Thus \(J_t(\mathcal W_t)\), equipped with the transported inner product, is a reproducing kernel Hilbert space with reproducing kernel \(k_t\).

Since \(J_t(\mathcal W_t)\) is a reproducing kernel Hilbert space with reproducing kernel \(k_t\), and \(\mathcal H_t\) is uniquely determined by this kernel, it follows that \(J_t(\mathcal W_t)=\mathcal H_t\) isometrically. Consequently, for \(f=J_t g\),
\[
\|f\|_{\mathcal H_t}^2
=
\|g\|_{\mathcal W_t}^2
=
\int_{\widehat H}
|g(\theta)|^2\,e^{t\lambda_H(\theta)}\,d\mu_{\widehat H}(\theta),
\]
and uniqueness of \(g\) follows from injectivity of \(J_t\).
\end{proof}

Fix \(s,t>0\). For all \(x,y\in H\), we have
\begin{align*}
k_{s+t}(x,y)
&=
p_{s+t}^\psi(x-y) \\
&=
\sum_{z\in H} p_s^\psi(x-z)p_t^\psi(z-y) \\
&=
\sum_{z\in H} k_s(x,z)\,k_t(z,y),
\end{align*}
where the series converges absolutely since \(p_s^\psi,p_t^\psi\in \ell^1(H)\), so their convolution is absolutely summable. This is the identity $P_{s+t}=P_sP_t$ at the level of kernels.

If \(0<s<t\), then \(e^{s\lambda_H(\theta)}\le e^{t\lambda_H(\theta)}\) for all \(\theta\in\widehat H\). By Lemma~\ref{lem:spectral-Ht}, it follows that
\begin{align*}
\|f\|_{\mathcal H_s}^2
&=
\int_{\widehat H}
|g(\theta)|^2\,e^{s\lambda_H(\theta)}\,d\mu_{\widehat H}(\theta) \\
&\le
\int_{\widehat H}
|g(\theta)|^2\,e^{t\lambda_H(\theta)}\,d\mu_{\widehat H}(\theta) \\
&=
\|f\|_{\mathcal H_t}^2
\end{align*}
for all \(f\in \mathcal H_t\), where \(g\in\mathcal W_t\) is the unique element such that \(f=J_t g\). In particular, the inclusion \(\mathcal H_t\hookrightarrow \mathcal H_s\) is continuous whenever \(0<s<t\).

For $t>0$ and $R>0$, define
\begin{equation*}
k_{t,R}(x,y)
=
\int_{\widehat H}
\theta(x-y)\,e^{-t\lambda_{H,R}(\theta)}\,d\mu_{\widehat H}(\theta)
\end{equation*}
for $x,y\in H$. Since $\nu_R\le \nu$ and $\nu\in \ell^1(H\setminus\{0\})$, the series defining $\lambda_{H,R}$ converges absolutely and consists of nonnegative terms. For each \(\kappa\in H\setminus\{0\}\), the function \(\theta\mapsto 1-\Re\,\theta(\kappa)\) is continuous and negative definite on \(\widehat H\). It follows that \(\lambda_{H,R}\), being a nonnegative summable linear combination of such functions, is continuous and negative definite. Hence \(e^{-t\lambda_{H,R}}\) is positive definite on \(\widehat H\), and therefore \(k_{t,R}\) is real, symmetric, and positive definite on \(H\). Let \(\mathcal H_{t,R}\) denote the associated reproducing kernel Hilbert space.

Since $\nu \in \ell^1(H \setminus \{0\})$, we have
$0 \le \lambda_H(\theta) \le 2\sum_{\kappa \in H \setminus \{0\}} \nu(\kappa) < \infty$
for all $\theta \in \widehat H$.

\begin{lemma}\label{lem:trunc-norm}
Let \(t>0\) and let \(f\in c_{00}(H)\). Then
\[
\|f\|_{\mathcal H_{t,R}}^2 \uparrow \|f\|_{\mathcal H_t}^2
\quad \text{as } R\to\infty.
\]
\end{lemma}

\begin{proof}
Fix \(t>0\) and \(f\in c_{00}(H)\). Since \(f\) has finite support, its Fourier transform \(\widehat f\) is a trigonometric polynomial on \(\widehat H\), and hence \(\widehat f\in L^2(\widehat H,\mu_{\widehat H})\). Since $\nu \in \ell^1(H \setminus \{0\})$, the function $\lambda_H$ is bounded.
It follows that
\[
|\widehat f(\theta)|^2 e^{t\lambda_H(\theta)}
\le
e^{t\|\lambda_H\|_\infty} |\widehat f(\theta)|^2 \in L^1(\widehat H,\mu_{\widehat H}),
\]
so \(f\in \mathcal H_t\). By Lemma~\ref{lem:spectral-Ht}, we have
\[
\|f\|_{\mathcal H_t}^2
=
\int_{\widehat H}
|\widehat f(\theta)|^2 e^{t\lambda_H(\theta)}\,d\mu_{\widehat H}.
\]

We next establish the corresponding representation for $\|f\|_{\mathcal H_{t,R}}$. The kernel $k_{t,R}$ is defined by the Fourier multiplier $e^{-t\lambda_{H,R}}$, and the proof of Lemma~\ref{lem:spectral-Ht} applies with $\lambda_{H,R}$ in place of $\lambda_H$, once one notes that $\lambda_{H,R}$ is bounded and continuous. Since $\nu_R\le \nu$, we have
\begin{align*}
\lambda_{H,R}(\theta)
&\le 2\sum_{\kappa\in H\setminus\{0\}}\nu_R(\kappa) \\
&\le 2\sum_{\kappa\in H\setminus\{0\}}\nu(\kappa)
< \infty
\end{align*}
for $\theta\in\widehat H$, and hence the same bound as above shows that $|\widehat f|^2 e^{t\lambda_{H,R}}$ is integrable. Consequently,
\[
\|f\|_{\mathcal H_{t,R}}^2
=
\int_{\widehat H}
|\widehat f(\theta)|^2 e^{t\lambda_{H,R}(\theta)}\,d\mu_{\widehat H}.
\]

By Lemma~\ref{lem:lambdaHR-monotone}, we have \(\lambda_{H,R}(\theta)\uparrow \lambda_H(\theta)\) for all \(\theta\in\widehat H\) as \(R\to\infty\). Therefore, $|\widehat f(\theta)|^2 e^{t\lambda_{H,R}(\theta)} \uparrow |\widehat f(\theta)|^2 e^{t\lambda_H(\theta)}$ for all $\theta\in\widehat H$, and the integrands are nonnegative. Since the limit is integrable, the monotone convergence theorem gives
\begin{align*}
\|f\|_{\mathcal H_{t,R}}^2 &= \int_{\widehat H} |\widehat f(\theta)|^2 e^{t\lambda_{H,R}(\theta)}\,d\mu_{\widehat H} \\
&\uparrow \int_{\widehat H} |\widehat f(\theta)|^2 e^{t\lambda_H(\theta)}\,d\mu_{\widehat H} \\
&= \|f\|_{\mathcal H_t}^2,
\end{align*}
which proves the claim.
\end{proof}

Fix $t>0$. Define $\Phi_t:H\to\mathcal H_t$ by $\Phi_t(x)=k_t(\cdot,x)$ for $x\in H$. For $\kappa\in H$, define $(\tau_\kappa f)(y)=f(y-\kappa)$. Since $p_t^\psi$ is translation invariant, $\Phi_t(x+\kappa)=\tau_\kappa\Phi_t(x)$ for all $x,\kappa\in H$.

The linear span of $\{k_t(\cdot,x):x\in H\}$ is dense in $\mathcal H_t$.

\section{Random Walks on Virtual Persistence Diagrams}
\label{sec:stoch-mc}

This section studies the heat semigroup on $H$ through the random walk $(X_t)_{t\ge0}$, its transition kernel $(p_t)_{t\ge0}$, and its L\'evy--Khintchine exponent $\lambda_H$ from Theorem~\ref{thm:lk-H}. The first several subsections present toy examples, while the final subsection contains the main result of the section. These toy examples analyze concrete quantities associated with the heat semigroup on $H$ and show them as probabilities, operator norms, and constants in the inequalities below, and are organized around four quantities associated with the random walk and its Fourier representation through $\lambda_H$.

The four quantities are
\[
p_t(0), \qquad
\sum_{h\in H} p_t(h)^2, \qquad
\int_{\widehat H}\lambda_H(\theta)e^{-t\lambda_H(\theta)}\,d\mu_{\widehat H}(\theta), \qquad
G_s(0,0),
\]
namely the return probability, the collision probability, the heat kernel energy, and the diagonal resolvent value. Subsection~\ref{subsec:heat-invariants} and Subsection~\ref{subsec:rkhs-interpretation} identify these quantities through the identities
\[
\|\mathrm{ev}_0\|_{\mathcal H_t^*}^2 = p_t(0),
\qquad
\|P_t\|_{\ell^2(H)\to\ell^\infty(H)}^2 = \sum_{h\in H} p_t(h)^2.
\]
The heat kernel energy and the diagonal resolvent value are the constants in the Lipschitz bound and the resolvent inequality in Theorem~\ref{thm:lip-rkhs-invariants} and Theorem~\ref{thm:sobolev-resolvent}. Theorem~\ref{thm:mass-tail} relates the L\'evy measure $\nu$ to tail bounds for $\rho(X_t,0)$. Theorem~\ref{thm:covering-high-density} bounds the covering numbers of the superlevel sets $\{h\in H : p_t(h)\ge \alpha\}$ in terms of the collision probability.

We introduce the truncated L\'evy measure $\nu_R$, the truncated exponent $\lambda_{H,R}$, and the associated convolution semigroups $(p_{t,R})_{t\ge0}$. By Corollary~\ref{cor:finite-activity-compound-poisson}, the semigroups $(p_{t,R})_{t\ge0}$ are compound Poisson, and Lemma~\ref{lem:lambdaHR-monotone} with Lemma~\ref{lem:trunc-norm} show that the corresponding return probabilities, collision probabilities, heat kernel energies, and resolvent values converge monotonically to their full counterparts. The final subsection is of a different nature and contains the main result of the section, namely the heat-scale majorization theorem. It passes from the heat kernels $p_t$, equivalently $e^{-t\lambda_H}$, to kernels defined by mixtures~\cite{berg1984harmonic}
\[
m_\eta(\lambda) = \int_{[0,\infty)} e^{-u\lambda}\,d\eta(u),
\]
applied to $\lambda_H$, and establishes heat-scale majorization under convex order of the measures $\eta$.

By Theorem~\ref{thm:lk-H} and the L\'evy--Khintchine representation \eqref{eq:lk-series}, the convolution semigroup $(p_t)_{t\ge0}$ on the subgroup $H$ has continuous negative definite symbol $\lambda_H:\widehat H\to[0,\infty)$ of the form
\begin{align*}
\lambda_H(\theta)
&=
\sum_{\kappa\in H\setminus\{0\}}
\nu(\kappa)\bigl(1-\Re\,\theta(\kappa)\bigr) \\
&=
\sup_{\substack{J \subset H\setminus\{0\} \\ |J|<\infty}}
\sum_{\kappa\in J}
\nu(\kappa)\bigl(1-\Re\,\theta(\kappa)\bigr)
\end{align*}
for $\theta\in\widehat H$, with symmetric L\'evy measure $\nu$ defined above~\cite[Theorem~18.19]{BergForst1975}. On $\ell^2(H)$ we define the associated Dirichlet form $\mathcal E_H$ by
\begin{equation*}
\mathcal E_H(f,g)
=
\int_{\widehat H}
\lambda_H(\theta)\,\widehat f(\theta)\,\overline{\widehat g(\theta)}\,
d\mu_{\widehat H}(\theta)
\end{equation*}
for $f,g\in\mathcal D(\mathcal E_H)$, with domain
\[ \mathcal D(\mathcal E_H) = \Bigl\{ f\in\ell^2(H) : \int_{\widehat H} \lambda_H(\theta)\,|\widehat f(\theta)|^2\, d\mu_{\widehat H}(\theta) <\infty \Bigr\}. \]
This is the quadratic form of the generator of $(P_t)_{t\ge0}$ and gives the energy term in the resolvent and Sobolev--type bounds below.

\subsection{Heat kernel and resolvent invariants}
\label{subsec:heat-invariants}

We define heat kernel and resolvent quantities for the symmetric random walk $(X_t)_{t\ge0}$ on $H$, express them in terms of the transition kernel $p_t$, and identify them with the analytic objects determined by the generator and Dirichlet form $\mathcal E_H$.

\subsubsection{Return probability}

For $t>0$, the \emph{return probability} of the random walk is
\[
 \mathbb P(X_t=0)=p_t(0).
\]
This is the on--diagonal heat kernel. By Fourier inversion on the countable
abelian group $H$,
\[
 p_t(0)
 =
 \int_{\widehat H} e^{-t\lambda_H(\theta)}\,d\mu_{\widehat H}(\theta).
\]

The rate of decay of the return probability is given by its right derivative. For
$t>0$,
\[
 -\frac{d}{dt}p_t(0)
 =
 \int_{\widehat H}
 \lambda_H(\theta)\,e^{-t\lambda_H(\theta)}\,
 d\mu_{\widehat H}(\theta).
\]
Differentiation under the integral is justified by the bound
$\lambda e^{-t\lambda}\le (et)^{-1}$ for all $\lambda\ge0$ and $t>0$.

Let $L$ denote the nonnegative self--adjoint generator associated with the Dirichlet
form $\mathcal E_H$. The same quantity admits the equivalent representations
\begin{align*}
 -\frac{d}{dt}p_t(0)
 &=
 \langle \delta_0,\,L e^{-tL}\delta_0\rangle_{\ell^2(H)} \\
 &=
 \mathcal E_H\!\left(e^{-tL/2}\delta_0,\,e^{-tL/2}\delta_0\right).
\end{align*}
So the decay of the return probability coincides with the Dirichlet energy of the heat kernel at time $t$.

It is convenient to normalize the heat kernel weight and consider the
probability measure on $\widehat H$ given by
\[
 d\mu_t(\theta)
 =
 \frac{e^{-t\lambda_H(\theta)}}{\int_{\widehat H} e^{-t\lambda_H}\,d\mu_{\widehat H}}
 \,d\mu_{\widehat H}(\theta).
\]
With this notation,
\[
 \mathbb E_{\mu_t}[\lambda_H]
 =
 \frac{\displaystyle\int_{\widehat H}
 \lambda_H(\theta)\,e^{-t\lambda_H(\theta)}\,d\mu_{\widehat H}(\theta)}
 {\displaystyle\int_{\widehat H}
 e^{-t\lambda_H(\theta)}\,d\mu_{\widehat H}(\theta)}
\]
is the mean spectral energy under the heat kernel weight at time $t$.

\subsubsection{Collision probability}

Let $X_t$ and $X_t'$ be two independent copies of the random walk on $H$, both
started at the identity. The \emph{collision probability} at time $t>0$ is
\[
 \mathbb P(X_t=X_t').
\]
By independence and translation invariance,
\begin{align*}
 \mathbb P(X_t=X_t')
 &=
 \sum_{h\in H} p_t(h)^2 \\
 &=
 \|p_t\|_{\ell^2(H)}^2.
\end{align*}
Plancherel's theorem gives the spectral representation
\[ \mathbb P(X_t=X_t') = \int_{\widehat H} e^{-2t\lambda_H(\theta)}\,d\mu_{\widehat H}(\theta). \]

The collision probability controls the $L^2\!\to\!L^\infty$ behavior of the
semigroup. For every $f\in\ell^2(H)$ and $t>0$,
\[ \|P_t f\|_\infty \le \mathbb P(X_t=X_t')^{1/2}\,\|f\|_2, \]
and the constant is sharp.

\subsubsection{Resolvent}

Let $L$ be the generator associated to $\mathcal E_H$. For $s>0$, define the
resolvent operator
\[ R_s = (s+L)^{-1} \]
and its diagonal kernel
\[ G_s(0,0) = \langle \delta_0,\,R_s\delta_0\rangle_{\ell^2(H)}. \]
This quantity admits the occupation--time representation
\begin{align*}
 G_s(0,0) &= \mathbb E\!\left[ \int_0^\infty e^{-st}\,\mathbf 1_{\{X_t=0\}}\,dt \right] \\
 &= \int_0^\infty e^{-st}\,p_t(0)\,dt.
\end{align*}

By the spectral theorem, the same quantity can be written equivalently as
\[ G_s(0,0) = \int_{\widehat H} \frac{1}{s+\lambda_H(\theta)}\,d\mu_{\widehat H}(\theta). \]

\subsection{Mass and covering inequalities}
\label{subsec:mass-covering}

This subsection studies the distribution of the heat kernel $p_t$ in $(H,\rho)$. We define the mass functional $\mathcal M$, prove $\rho(g,0)\le\mathcal M(g)$, and obtain tail bounds for $\mathcal M(X_t)$ from the L\'evy measure $\nu$. We also bound the covering numbers of the superlevel sets $\{h\in H:p_t(h)\ge\alpha\}$ using the quantity $\sum_{h\in H}p_t(h)^2$.

\begin{definition}\label{def:mass}
For $g\in K(X,A)$, write
\[
 g
 =
 \sum_{u\in X/A\setminus\{[A]\}} n_u\,e_u
\]
with $n_u\in\mathbb Z$ and only finitely many nonzero coefficients. The \emph{mass} of $g$ is defined by
\[ \mathcal M(g) = \sum_{u\in X/A\setminus\{[A]\}} |n_u|\,\overline d_1(u,[A]). \]
For $\alpha\in D(X,A)$, this coincides with $\mathcal M(\alpha)=\sum_{x\in\alpha}\overline d_1(x,[A])$.
\end{definition}

The functional $\mathcal M$ is the weighted $\ell^1$--norm of the coefficient vector $(n_u)_u$, with weights given by the distance to the basepoint $[A]$. In particular, $\mathcal M$ is homogeneous and subadditive, and $\mathcal M(g)=0$ if and only if $g=0$, since $\overline d_1(u,[A])>0$ for $u\neq[A]$.

\begin{lemma}\label{lem:rho-mass}
For every $g\in K(X,A)$,
\[ \rho(g,0)\ \le\ \mathcal M(g). \]
\end{lemma}

\begin{proof}
Write $g=\alpha-\beta$ with $\alpha,\beta\in D(X,A)$. By definition of $\rho$,
\begin{align*}
 \rho(g,0) &= \rho(\alpha-\beta,0) \\
 &= W_1(\alpha,\beta).
\end{align*}
Let $\gamma(u)=\min\{\alpha(u),\beta(u)\}$ for $u\in X/A$, and write $\alpha=\gamma+\alpha'$, $\beta=\gamma+\beta'$, so that $\alpha'$ and $\beta'$ have disjoint support. Then $W_1(\alpha,\beta)=W_1(\alpha',\beta')$. If $g=\sum_u n_u e_u$, we have
\[ \alpha'(u)=\max(n_u,0), \qquad \beta'(u)=\max(-n_u,0). \]
Consider the admissible matching that sends every point of $\alpha'$ and $\beta'$ to the basepoint $[A]\in X/A$. Its total cost is
\[
 \sum_u \alpha'(u)\,\overline d_1(u,[A]) +
 \sum_u \beta'(u)\,\overline d_1(u,[A]) =
 \sum_u |n_u|\,\overline d_1(u,[A]) =
 \mathcal M(g).
\]
Since $W_1(\alpha',\beta')$ is the infimum of the costs of all admissible matchings, we obtain $W_1(\alpha',\beta')\le\mathcal M(g)$, and hence $\rho(g,0)\le\mathcal M(g)$.
\end{proof}

Lemma~\ref{lem:rho-mass} shows that $\mathcal M$ is a coarse geometric control function for the VPD metric: tail bounds for $\mathcal M(X_t)$ immediately imply tail bounds for $\rho(X_t,0)$.

We now quantify the probability that the random walk $X_t$ has large mass.

\begin{theorem}\label{thm:mass-tail}
Let $(X_t)_{t\ge0}$ be the L\'evy process on $H$ with convolution semigroup $(p_t)_{t\ge0}$ and L\'evy measure $\nu$, started at $X_0=0$. Then for every $t>0$ and $R>0$,
\[ \mathbb P\bigl(\mathcal M(X_t)>R\bigr) \;\le\; t\,\nu\{\kappa\in H\setminus\{0\} : \mathcal M(\kappa)>R\} \;+\; \frac{t}{R} \int_{\{\kappa\in H\setminus\{0\} : \mathcal M(\kappa)\le R\}} \mathcal M(\kappa)\,d\nu(\kappa). \]
In particular, by Lemma~\ref{lem:rho-mass}, the same bound holds with $\mathcal M(X_t)$ replaced by $\rho(X_t,0)$.
\end{theorem}

\begin{proof}
Fix $R>0$. Decompose the jumps of $X_t$ into \emph{large} and \emph{small} according to the mass threshold $R$: a jump $\kappa$ is called large if
$\mathcal M(\kappa)>R$ and small if $\mathcal M(\kappa)\le R$.

Let $N_t^{>R}$ denote the number of large jumps up to time $t$. By the L\'evy--Khintchine construction, $N_t^{>R}$ is a Poisson random variable with mean
\[
 t\,\nu\{\kappa\in H\setminus\{0\} : \mathcal M(\kappa)>R\}.
\]
Therefore,
\begin{align*}
 \mathbb P\bigl(N_t^{>R}\ge1\bigr)
 &=
 1-\exp\!\bigl(-t\,\nu\{\mathcal M>R\}\bigr) \\
 &\le
 t\,\nu\{\mathcal M>R\}.
\end{align*}

Next, let $X_t^{\le R}$ denote the process obtained from $X_t$ by retaining only the small jumps. Then $X_t^{\le R}$ is again a L\'evy process on $H$, with L\'evy measure $\nu^{\le R}=\nu|_{\{\mathcal M\le R\}}$. If no large jump occurs up to time $t$, then $X_t=X_t^{\le R}$.

By subadditivity of $\mathcal M$ along jumps,
\[
 \mathcal M(X_t^{\le R})
 \;\le\;
 \sum_{j=1}^{N_t^{\le R}} \mathcal M(J_j),
\]
where $\{J_j\}_{j\ge1}$ are the small jumps and $N_t^{\le R}$ is their total number up to time $t$. Taking expectations and using the independence of jumps gives
\[ \mathbb E\bigl[\mathcal M(X_t^{\le R})\bigr] = t\int_{\{\mathcal M\le R\}} \mathcal M(\kappa)\,d\nu(\kappa). \]
By Markov's inequality,
\[ \mathbb P\bigl(\mathcal M(X_t^{\le R})>R\bigr) \le \frac{t}{R} \int_{\{\mathcal M\le R\}} \mathcal M(\kappa)\,d\nu(\kappa). \]

The event $\{\mathcal M(X_t)>R\}$ can occur only if at least one large jump occurs or if the accumulated contribution of small jumps exceeds $R$. Hence
\[ \{\mathcal M(X_t)>R\} \subset \{N_t^{>R}\ge1\} \;\cup\; \{\mathcal M(X_t^{\le R})>R\}, \]
and the stated bound follows by combining the two estimates above. The final claim follows from $\rho(X_t,0)\le\mathcal M(X_t)$.
\end{proof}

We next turn to covering numbers for subsets of the group $H$.

\begin{definition}\label{def:covering-number}
For a subset $S\subset K(X,A)$ and $\varepsilon>0$, the \emph{covering number} $N(S,\varepsilon)$ is the smallest integer $m$ such that $S$ can be covered by $m$ open $\rho$--balls of radius $\varepsilon$.
\end{definition}

Translation invariance of $\rho$ implies $N(S+g,\varepsilon)=N(S,\varepsilon)$ for all $g\in K(X,A)$. Since $(K(X,A),\rho)$ is discrete and $\rho$ is translation invariant, there exists $\varepsilon_0>0$ such that every open $\rho$--ball of radius $\varepsilon<\varepsilon_0$ contains at most one point of $H$. Consequently, if $S\subset H$ is finite and $\varepsilon\in(0,\varepsilon_0)$, then $N(S,\varepsilon)=|S|$.

For $t>0$ and $\alpha>0$, define the heat kernel superlevel set $A_t(\alpha)=\{h\in H : p_t(h)\ge\alpha\}$.

\begin{theorem}
\label{thm:covering-high-density}
Let $t>0$ and $\alpha>0$. There exists $\varepsilon_0>0$ such that for every $\varepsilon\in(0,\varepsilon_0)$,
\begin{align*}
 N\bigl(A_t(\alpha),\varepsilon\bigr)
 &\le
 \min\left\{
 \frac{1}{\alpha},
 \frac{\mathbb P(X_t=X_t')}{\alpha^2}
 \right\} \\
 &=
 \min\left\{
 \frac{1}{\alpha},
 \frac{1}{\alpha^2}\sum_{h\in H}p_t(h)^2
 \right\}.
\end{align*}
where $X_t'$ is an independent copy of $X_t$ started at the origin.
\end{theorem}

\begin{proof}
Since $p_t$ is a probability mass function on $H$,
\begin{align*}
 1
 &= \sum_{h\in H} p_t(h) \\
 &\ge \sum_{h\in A_t(\alpha)} p_t(h) \\
 &\ge \alpha\,|A_t(\alpha)|.
\end{align*}
and hence $|A_t(\alpha)|\le 1/\alpha$.

On the other hand, by Cauchy--Schwarz,
\[
 \sum_{h\in A_t(\alpha)} p_t(h)
 \le
 |A_t(\alpha)|^{1/2}
 \Bigl(\sum_{h\in H} p_t(h)^2\Bigr)^{1/2}.
\]
Combining this with the lower bound
$\sum_{h\in A_t(\alpha)} p_t(h)\ge\alpha |A_t(\alpha)|$ gives
\[
 |A_t(\alpha)|
 \le
 \frac{1}{\alpha^2}
 \sum_{h\in H} p_t(h)^2.
\]
Since $\sum_{h\in H} p_t(h)^2=\mathbb P(X_t=X_t')$, the stated cardinality bounds follow. For $\varepsilon\in(0,\varepsilon_0)$ we have $N(A_t(\alpha),\varepsilon)=|A_t(\alpha)|$, and the covering inequality follows.
\end{proof}

The mass tail inequality and the covering bounds obtained above give localized invariants of the heat kernel on the VPD group $(H,\rho)$. The mass functional $\mathcal M$ gives a control on the typical $\rho$--distance of the random walk $X_t$ from the identity in terms of the L\'evy measure, and the covering estimates bound the size of heat kernel superlevel sets in terms of the collision probability $\mathbb P(X_t=X_t')$.

\subsection{Lipschitz seminorms}
\label{subsec:lipschitz-rkhs}

We derive Lipschitz bounds for functions in the heat kernel reproducing kernel Hilbert space on $H$. In contrast with the finite rank case treated in \cite{fanning2025reproducingkernelhilbertspaces}, the proof combines the comparison between characters on $(K(X,A),\rho)$ and phase differences on $(X/A,\overline d_1)$ with the localization of characters under the heat kernel, expressed through the heat kernel invariants from Subsection~\ref{subsec:heat-invariants}.

\subsubsection{Characters and Lipschitz seminorms}

For each $x\in X\setminus A$, let $e_x\in K(X,A)$ denote the class of the one--point diagram at $x$, and set $e_{[A]}=0$. Given a character $\chi\in\widehat K(X,A)$, define the associated phase map
\[
 \phi_\chi\colon X/A \longrightarrow \mathbb R/2\pi\mathbb Z,
 \qquad
 \phi_\chi([A])=0,
 \quad
 \phi_\chi(x)=\arg(\chi(e_x)).
\]
Equip $\mathbb R/2\pi\mathbb Z$ with its geodesic distance $\operatorname{dist}\in[0,\pi]$, and define the Lipschitz seminorm
\[
 \mathrm{Lip}_{\overline d_1}(\phi_\chi)
 =
 \sup_{u\neq v}
 \frac{\operatorname{dist}(\phi_\chi(u),\phi_\chi(v))}
 {\overline d_1(u,v)}
 \in [0,\infty].
\]

\begin{lemma}
\label{lem:char-lip-comparison-uniform}
For every character $\chi\in\widehat K(X,A)$,
\[ \frac{2}{\pi}\,\mathrm{Lip}_{\overline d_1}(\phi_\chi) \;\le\; \mathrm{Lip}_\rho(\chi) \;\le\; \mathrm{Lip}_{\overline d_1}(\phi_\chi).
\]
\end{lemma}

\begin{proof}
Translation invariance of $\rho$ gives
\[
 \mathrm{Lip}_\rho(\chi) = \sup_{\gamma\neq0} \frac{|\chi(\gamma)-1|}{\rho(\gamma,0)}.
\]

Assume $\mathrm{Lip}_{\overline d_1}(\phi_\chi)<\infty$ and fix $\gamma\in K(X,A)\setminus\{0\}$. Write $\gamma=\alpha-\beta$ with $\alpha,\beta\in D(X,A)$ finite diagrams, and let $\sigma$ be an optimal matching for $W_1(\alpha,\beta)$ with multiplicities $m_{x,y}\ge0$. Then
\begin{align*}
\rho(\gamma,0)
&= W_1(\alpha,\beta) \\
&= \sum_{x,y} m_{x,y}\,\overline d_1(x,y).
\end{align*}
For each matched pair $(x,y)$ choose $\delta_{x,y}\in[-\pi,\pi]$ such that $e^{i\delta_{x,y}}=e^{i(\phi_\chi(x)-\phi_\chi(y))}$ and $|\delta_{x,y}|=\operatorname{dist}(\phi_\chi(x),\phi_\chi(y))$. Then
\begin{align*}
|\chi(\gamma)-1|
&\le \sum_{x,y} m_{x,y}\,|\delta_{x,y}| \\
&\le \mathrm{Lip}_{\overline d_1}(\phi_\chi)\,\rho(\gamma,0).
\end{align*}

If $\mathrm{Lip}_{\overline d_1}(\phi_\chi)=0$ then $\chi$ is constant and the claim is trivial. Otherwise fix $\varepsilon>0$ and choose $u\neq v$ such that
\[ \frac{\operatorname{dist}(\phi_\chi(u),\phi_\chi(v))} {\overline d_1(u,v)} \ge \mathrm{Lip}_{\overline d_1}(\phi_\chi)-\varepsilon. \]
Set $\gamma=e_u-e_v$ and $\delta=\operatorname{dist}(\phi_\chi(u),\phi_\chi(v))\in(0,\pi]$. Then $\rho(\gamma,0)=\overline d_1(u,v)$ and
\begin{align*}
|\chi(\gamma)-1|
&= |e^{i\delta}-1| \\
&= 2\sin(\delta/2).
\end{align*}
Since $\sin t\ge (2/\pi)t$ for $t\in[0,\pi/2]$, this gives
\[
 |\chi(\gamma)-1|
 \ge
 \frac{2}{\pi}\,
 \bigl(\mathrm{Lip}_{\overline d_1}(\phi_\chi)-\varepsilon\bigr)\,
 \rho(\gamma,0),
\]
and letting $\varepsilon\downarrow0$ gives the result.
\end{proof}

\subsubsection{Spectral Lipschitz bound for heat kernel RKHS functions}

We combine Lemma~\ref{lem:char-lip-comparison-uniform} with Corollary~\ref{cor:heat-lip-spectral} and the heat kernel representation of $\mathcal H_t$ to obtain a Lipschitz bound on $H$ in terms of the heat kernel invariants from Subsection~\ref{subsec:heat-invariants}.

\begin{theorem}
\label{thm:lip-rkhs-invariants}
For every $t>0$ and every $f\in\mathcal H_t$,
\begin{align*}
 \mathrm{Lip}_\rho(f)
 &\le
 \|f\|_{\mathcal H_t}\,
 \bigl(-\tfrac{d}{dt}p_t(0)\bigr)^{1/2} \\
 &=
 \|f\|_{\mathcal H_t}\,
 \bigl(\mathbb E_{\mu_t}[\lambda_H]\bigr)^{1/2}\,p_t(0)^{1/2}.
\end{align*}
\end{theorem}

\begin{proof}
Fix $\gamma\in H$. By definition of $\mathrm{Lip}_\rho$ and translation invariance of $\rho$,
\[
 \mathrm{Lip}_\rho(f)
 =
 \sup_{\gamma\neq0}
 \frac{|f(\gamma)-f(0)|}{\rho(\gamma,0)}.
\]

Let $F\subset X\setminus A$ be a finite subset such that $\gamma\in\bigoplus_{x\in F}\mathbb Z e_x$. By Lemma~\ref{lem:projection}, the restriction of $(P_t)$ to $\ell^2\bigl(\bigoplus_{x\in F}\mathbb Z e_x\bigr)$ coincides with the finite--rank semigroup constructed in \cite{fanning2025reproducingkernelhilbertspaces}, with graph metric induced directly by $\overline d_1$.

Restricting the reproducing kernel to this finite subgroup and applying the finite--rank spectral Lipschitz estimate \cite[Corollary~2]{fanning2025reproducingkernelhilbertspaces} gives
\[
 |f(\gamma)-f(0)|
 \le
 \|f\|_{\mathcal H_t}
 \left(
 \int_{\widehat H}
 |\theta(\gamma)-1|^2\,e^{-t\lambda_H(\theta)}\,
 d\mu_{\widehat H}(\theta)
 \right)^{1/2}.
\]
Using Lemma~\ref{lem:char-lip-comparison-uniform} and the definition of $\lambda_H(\theta)$ gives
\[
 |\theta(\gamma)-1|
 \le
 \mathrm{Lip}_\rho(\theta)\,\rho(\gamma,0),
 \qquad
 \mathrm{Lip}_\rho(\theta)^2
 \le
 \lambda_H(\theta),
\]
and hence
\[
 |f(\gamma)-f(0)|
 \le
 \|f\|_{\mathcal H_t}\,
 \rho(\gamma,0)
 \left(
 \int_{\widehat H}
 \lambda_H(\theta)\,e^{-t\lambda_H(\theta)}\,
 d\mu_{\widehat H}(\theta)
 \right)^{1/2}.
\]
Dividing by $\rho(\gamma,0)$ and taking the supremum over $\gamma\neq0$ yields the claim. The alternative expressions follow from Subsection~\ref{subsec:heat-invariants}.
\end{proof}

\subsection{A Sobolev-type inequality with resolvent constant}
\label{subsec:sobolev-resolvent}

Subsection~\ref{subsec:heat-invariants} introduced the resolvent diagonal $G_s(0,0)$ associated with the generator $L$ of the random walk. We show that this quantity controls the sharp conversion of $\ell^2$ mass and Dirichlet energy into pointwise control.

\begin{theorem}
\label{thm:sobolev-resolvent}
For every $s>0$ and every $f\in\ell^2(H)$ with finite Dirichlet energy $\mathcal E_H(f,f)$,
\[
 \|f\|_\infty^2
 \le
 G_s(0,0)\,
 \bigl(s\,\|f\|_2^2+\mathcal E_H(f,f)\bigr).
\]
The constant $G_s(0,0)$ is optimal.
\end{theorem}

\begin{proof}
Fix $s>0$ and $f\in\ell^2(H)$ with $\mathcal E_H(f,f)<\infty$. For each $x\in H$, Fourier inversion gives
\[
 f(x)
 =
 \int_{\widehat H}\widehat f(\theta)\,\theta(x)\,
 d\mu_{\widehat H}(\theta).
\]
Applying the Cauchy--Schwarz inequality with weight $s+\lambda_H(\theta)$ gives
\begin{align*}
 |f(x)|^2
 &\le
 \Biggl(
 \int_{\widehat H}
 (s+\lambda_H(\theta))\,|\widehat f(\theta)|^2\,
 d\mu_{\widehat H}(\theta)
 \Biggr)
 \Biggl(
 \int_{\widehat H}
 \frac{1}{s+\lambda_H(\theta)}\,
 d\mu_{\widehat H}(\theta)
 \Biggr).
\end{align*}
By the definition of the Dirichlet form $\mathcal E_H$ in terms of $\lambda_H$, the first factor equals $s\,\|f\|_2^2+\mathcal E_H(f,f)$, and by the spectral representation of $G_s(0,0)$ in Subsection~\ref{subsec:heat-invariants}, the second factor equals $G_s(0,0)$. Hence
\[
 |f(x)|^2
 \le
 G_s(0,0)\,
 \bigl(s\,\|f\|_2^2+\mathcal E_H(f,f)\bigr)
\]
for every $x\in H$. Taking the supremum over $x$ yields the inequality.

For optimality, equip $\mathcal D(\mathcal E_H)$ with the norm
\[
 \|f\|_s^2
 =
 s\,\|f\|_2^2+\mathcal E_H(f,f),
\]
and let $\mathcal H_s$ denote the Hilbert space completion. Evaluation at the origin defines a bounded linear functional on $\mathcal H_s$, whose squared operator norm is
\[
 \sup_{f\neq0}\frac{|f(0)|^2}{\|f\|_s^2}.
\]
By the Riesz representation theorem and the definition of the resolvent, this quantity equals $G_s(0,0)$. Therefore the constant in the inequality cannot be improved.
\end{proof}

\subsection{RKHS interpretation of the random--walk invariants}
\label{subsec:rkhs-interpretation}

The random--walk invariants introduced in Subsection~\ref{subsec:heat-invariants} have canonical representations as squared operator norms of linear maps determined by the RKHS and Dirichlet structures constructed earlier in the paper.

In the heat kernel RKHS $\mathcal H_t$ (Section~\ref{sec:heat-rkhs}), evaluation at the identity defines a bounded linear functional $\mathrm{ev}_0\colon\mathcal H_t\to\mathbb R$. By the reproducing property, its operator norm satisfies
\begin{align*}
\|\mathrm{ev}_0\|_{\mathcal H_t^*}^2
 &= k_t(0,0) \\
 &= p_t(0).
\end{align*}
so the return probability is exactly the squared norm of evaluation in $\mathcal H_t$.

The heat semigroup $P_t$ acts as a bounded convolution operator on $\ell^2(H)$. Its $\ell^2(H)\to\ell^\infty(H)$ operator norm is characterized by the usual extremal property defining operator norms, and one has
\begin{align*}
\|P_t\|_{\ell^2(H)\to\ell^\infty(H)}^2
 &= \|p_t\|_{\ell^2(H)}^2 \\
 &= \mathbb{P}(X_t = X_t').
\end{align*}
identifying the collision probability with the squared smoothing norm of $P_t$.

Let $L$ denote the generator of $(P_t)_{t\ge0}$.
\begin{align*}
-\frac{d}{dt} p_t(0) &= \langle \delta_0,\, L e^{-tL} \delta_0 \rangle_{\ell^2(H)} \\
&= \bigl\| L^{1/2} e^{-tL/2} \delta_0 \bigr\|_{\ell^2(H)}^2.
\end{align*}
Accordingly, $-p_t'(0)$ is the Dirichlet energy of the canonical kernel section $e^{-tL/2}\delta_0$, and is precisely the scalar controlling the Lipschitz estimate in Theorem~\ref{thm:lip-rkhs-invariants}.

In the Sobolev space $\mathcal H_s$ associated with the Dirichlet form $\mathcal E_H$ (Subsection~\ref{subsec:sobolev-resolvent}), evaluation at the identity is a bounded linear functional whose norm is fixed by the reproducing property of the resolvent kernel. One has
\[
 \|\mathrm{ev}_0\|_{\mathcal H_s^*}^2 = G_s(0,0),
\]
so the diagonal resolvent value is the squared norm of evaluation in $\mathcal H_s$.

\[
\begin{array}{c|c|c}
\text{Invariant}
&
\text{Spectral form}
&
\text{Operator norm form}
\\
\hline
p_t(0)
&
\displaystyle\int_{\widehat H} e^{-t\lambda_H}\,d\mu_{\widehat H}
&
\|\mathrm{ev}_0\|_{\mathcal H_t^*}^2
\\[1ex]
-\,p_t'(0)
&
\displaystyle\int_{\widehat H} \lambda_H e^{-t\lambda_H}\,d\mu_{\widehat H}
&
\bigl\|L^{1/2}e^{-tL/2}\delta_0\bigr\|_{\ell^2(H)}^2
\\[1ex]
\|p_t\|_2^2
&
\displaystyle\int_{\widehat H} e^{-2t\lambda_H}\,d\mu_{\widehat H}
&
\|P_t\|_{\ell^2(H)\to\ell^\infty(H)}^2
\\[1ex]
G_s(0,0)
&
\displaystyle\int_{\widehat H} \frac{1}{s+\lambda_H}\,d\mu_{\widehat H}
&
\|\mathrm{ev}_0\|_{\mathcal H_s^*}^2
\end{array}
\]

In each case, the invariant is uniquely characterized as the squared norm of a canonical linear map determined by the universal extremal property defining the corresponding Hilbert--space structure.

\subsection{Metric truncation}
\label{subsec:truncation-mc}

Throughout this subsection we work with the metric truncations $\nu_R$ of the L\'evy measure $\nu$ and with the associated L\'evy--Khintchine exponents $\lambda_{H,R}$ constructed in Section~\ref{sec:heat-rkhs}. Since $\nu_R$ is supported on a finite $\rho$--ball in the discrete group $H$, it has finite total mass. The corresponding symmetric convolution semigroup is therefore a finite--activity L\'evy process, which has an exact compound Poisson representation by Corollary~\ref{cor:finite-activity-compound-poisson}. We denote the resulting convolution semigroup by $(p_{t,R})_{t\ge0}$.

The truncated spectral quantities are obtained from the formulas of Subsection~\ref{subsec:heat-invariants} by replacing $\lambda_H$ with $\lambda_{H,R}$. For $t>0$ and $s>0$,
\[
 p_{t,R}(0)
 =
 \int_{\widehat H} e^{-t\lambda_{H,R}(\theta)}\,d\mu_{\widehat H}(\theta),
\]
\[
 -\frac{d}{dt}p_{t,R}(0)
 =
 \int_{\widehat H}
 \lambda_{H,R}(\theta)\,e^{-t\lambda_{H,R}(\theta)}\,
 d\mu_{\widehat H}(\theta),
\]
\[
 \|p_{t,R}\|_{\ell^2(H)}^2
 =
 \int_{\widehat H} e^{-2t\lambda_{H,R}(\theta)}\,d\mu_{\widehat H}(\theta),
\]
and
\[
 G_{s,R}(0,0)
 =
 \int_{\widehat H} \frac{1}{s+\lambda_{H,R}(\theta)}\,d\mu_{\widehat H}(\theta).
\]

By Lemma~\ref{lem:lambdaHR-monotone}, $\lambda_{H,R}(\theta)\uparrow\lambda_H(\theta)$ for every $\theta\in\widehat H$. The inequalities
\begin{align*}
 0 &\le e^{-t\lambda_{H,R}(\theta)} \le 1,
 &\qquad
 0 &\le e^{-2t\lambda_{H,R}(\theta)} \le 1,
 \label{eq:trunc-bounds-exp} \\
 \lambda e^{-t\lambda} &\le (et)^{-1},
 &\qquad
 0 &\le \frac{1}{s+\lambda_{H,R}(\theta)} \le \frac{1}{s},
\end{align*}
which hold for all $\theta\in\widehat H$ and all $\lambda\ge0$, imply by monotone and dominated convergence that, for fixed $t>0$ and $s>0$,
\[
 p_{t,R}(0)\downarrow p_t(0),
 \qquad
 \|p_{t,R}\|_{\ell^2(H)}^2\downarrow\|p_t\|_{\ell^2(H)}^2,
\]
\[
 -\frac{d}{dt}p_{t,R}(0)\longrightarrow -\frac{d}{dt}p_t(0),
 \qquad
 G_{s,R}(0,0)\longrightarrow G_s(0,0)
\]
as $R\to\infty$.

Since $\lambda_{H,R}\le\lambda_H$, each truncated quantity dominates its infinite--rank counterpart for fixed $R$. Substituting the truncated expressions into Theorems~\ref{thm:lip-rkhs-invariants} and~\ref{thm:sobolev-resolvent} therefore gives valid upper bounds on the corresponding Lipschitz and resolvent constants. These bounds are monotone in $R$ and converge to the sharp constants as the truncation radius increases. Metric truncation thus giving a controlled numerical method with which the constants controlling global regularity on the virtual persistence diagram group can be approximated from above.

\subsection{Heat--scale majorization}

The preceding subsections derived spectral quantities attached to the heat semigroup on $H$, including the return probability, collision probability, heat-kernel energy, and diagonal resolvent value in Subsection~\ref{subsec:heat-invariants}, the Lipschitz and Sobolev bounds in Theorems~\ref{thm:lip-rkhs-invariants} and~\ref{thm:sobolev-resolvent}, and the truncation scheme in Subsection~\ref{subsec:truncation-mc}. In each case, the relevant quantity is given by a spectral integral against a function of $\lambda_H$. We now pass from the pure heat weight $e^{-t\lambda_H}$ to the mixture $m_\eta(\lambda_H)$ and prove the main result of this subsection, Theorem~\ref{thm:heat-scale-majorization}. 

That theorem shows that convex ordering of heat scales induces a corresponding ordering of the associated kernels, reproducing kernel Hilbert spaces, semimetrics, and scalar invariants.

Let $\eta$ be a finite positive Borel measure on $[0,\infty)$. Define $m_\eta(\lambda)=\int_{[0,\infty)} e^{-u\lambda}\,d\eta(u)$ for $\lambda\ge 0$.

\begin{lemma}\label{lem:convexity-main}
For each $\lambda\ge 0$, the functions $u\mapsto e^{-u\lambda}$ and $u\mapsto \lambda e^{-u\lambda}$ are convex on $[0,\infty)$.
\end{lemma}

\begin{proof}
Fix $\lambda\ge 0$. Direct differentiation gives $\frac{d^2}{du^2}e^{-u\lambda}=\lambda^2 e^{-u\lambda}\ge 0$ and $\frac{d^2}{du^2}(\lambda e^{-u\lambda})=\lambda^3 e^{-u\lambda}\ge 0.$
\end{proof}

Lemma~\ref{lem:convexity-main} gives the convex test functions needed to compare the mixtures under convex order.

\begin{lemma}\label{lem:multiplier-order-main}
Let $\eta_1$ and $\eta_2$ be finite positive Borel measures on $[0,\infty)$ with finite first moments. If $\eta_1 \preceq_{\mathrm{cx}} \eta_2$, then $m_{\eta_1}(\lambda)\le m_{\eta_2}(\lambda)$ and $\lambda m_{\eta_1}(\lambda)\le \lambda m_{\eta_2}(\lambda)$ for all $\lambda\ge 0$.
\end{lemma}

\begin{proof}
Fix $\lambda\ge 0$. One has $m_{\eta_i}(\lambda)=\int e^{-u\lambda}\,d\eta_i(u)$ and $\lambda m_{\eta_i}(\lambda)=\int \lambda e^{-u\lambda}\,d\eta_i(u)$ for $i=1,2$. By Lemma~\ref{lem:convexity-main}, both integrands are convex, and the convex-order assumption gives the result.
\end{proof}

For such a measure $\eta$, we use the spectral weight $m_\eta(\lambda_H)$ to define a translation-invariant kernel on $H$ by
\begin{equation*}
K_\eta(g,h) = \int_{\widehat H} \theta(h-g)\,m_\eta(\lambda_H(\theta))\,d\mu_{\widehat H}(\theta)
\end{equation*}
for $g,h\in H$.

\begin{lemma}\label{lem:kernel-pd-main}
For every finite positive Borel measure $\eta$ on $[0,\infty)$, the kernel $K_\eta$ is positive definite on $H\times H$.
\end{lemma}

\begin{proof}
Let $(c_j)\subset\mathbb C$ and $(g_j)\subset H$ be finite. Then
\[
\sum_{i,j} \overline{c_i}c_j K_\eta(g_i,g_j)
=
\int_{\widehat H}
\left|\sum_j c_j \theta(g_j)\right|^2
m_\eta(\lambda_H(\theta))\,d\mu_{\widehat H}(\theta)\ge 0,
\]
which proves positive definiteness.
\end{proof}

We now use the multiplier ordering to compare the corresponding kernels.

\begin{lemma}\label{lem:kernel-order-main}
Let $\eta_1$ and $\eta_2$ be as in Lemma~\ref{lem:multiplier-order-main}. Then $K_{\eta_1}\preceq K_{\eta_2}$.
\end{lemma}

\begin{proof}
For finite $(c_j)$ and $(g_j)$,
\[
\sum_{i,j}\overline{c_i}c_j\bigl(K_{\eta_2}-K_{\eta_1}\bigr)(g_i,g_j)
=
\int_{\widehat H}
\left|\sum_j c_j\theta(g_j)\right|^2
\bigl(m_{\eta_2}-m_{\eta_1}\bigr)(\lambda_H(\theta))\,d\mu_{\widehat H}(\theta)\ge 0,
\]
by Lemma~\ref{lem:multiplier-order-main}.
\end{proof}

Write $\mathcal H_{K_\eta}$ for the reproducing kernel Hilbert space of $K_\eta$.The kernel ordering above yields the corresponding contractive embedding of these spaces.

\begin{lemma}\label{lem:rkhs-embedding-main2}
If $K_{\eta_1}\preceq K_{\eta_2}$, then $\mathcal H_{K_{\eta_1}}\hookrightarrow \mathcal H_{K_{\eta_2}}$ contractively~\cite{jorgensen2016extensions}.
\end{lemma}

\begin{proof}
For every finite family $(c_j,h_j)$,
\begin{align*}
\sum_{i,j} c_i\overline{c_j} K_{\eta_2}(h_j,h_i)
&=
\sum_{i,j} c_i\overline{c_j} K_{\eta_1}(h_j,h_i)
+
\sum_{i,j} c_i\overline{c_j} \bigl(K_{\eta_2}-K_{\eta_1}\bigr)(h_j,h_i) \\
&\ge
\sum_{i,j} c_i\overline{c_j} K_{\eta_1}(h_j,h_i),
\end{align*}
since $K_{\eta_2}-K_{\eta_1}$ is positive definite. Thus the identity map is contractive on the pre-Hilbert space of kernel sections and extends by completion to a contractive embedding $\mathcal H_{K_{\eta_1}}\hookrightarrow \mathcal H_{K_{\eta_2}}$.
\end{proof}

The kernel $K_\eta$ defines the semimetric $d_\eta$, and the spectral weight $m_\eta(\lambda_H)$ along with $K_\eta$ define the scalar quantities $A_\eta$ and $B_\eta$:
\begin{align*}
d_\eta(g,h) &= \bigl( K_\eta(g,g)+K_\eta(h,h)-2\Re K_\eta(g,h) \bigr)^{1/2}, \\
A_\eta &= \int_{\widehat H}\lambda_H(\theta)\,m_\eta(\lambda_H(\theta))\,d\mu_{\widehat H}(\theta), \\
B_\eta &= K_\eta(0,0) = \int_{\widehat H} m_\eta(\lambda_H(\theta))\,d\mu_{\widehat H}(\theta).
\end{align*}

\begin{theorem}\label{thm:heat-scale-majorization}
Let $\eta_1,\eta_2$ be finite positive Borel measures on $[0,\infty)$ with finite first moments. If $\eta_1 \preceq_{\mathrm{cx}} \eta_2$, then:
\begin{enumerate}
\item $K_{\eta_1}\preceq K_{\eta_2}$. In particular:
\begin{enumerate}
\item $\mathcal H_{K_{\eta_1}}\hookrightarrow \mathcal H_{K_{\eta_2}}$ contractively,
\item $d_{\eta_1}(g,h)\le d_{\eta_2}(g,h)$ for all $g,h\in H$,
\item $B_{\eta_1}\le B_{\eta_2}$,
\end{enumerate}
\item If $A_{\eta_2}<\infty$, then:
\begin{enumerate}
\item $A_{\eta_1}<\infty$,
\item $A_{\eta_1}\le A_{\eta_2}$,
\item $d_{\eta_2}(g,h)\le A_{\eta_2}^{1/2}\rho(g,h)$ for all $g,h\in H$.
\end{enumerate}
\end{enumerate}
\end{theorem}

\begin{proof}
By Lemma~\ref{lem:kernel-order-main}, one has $K_{\eta_1}\preceq K_{\eta_2}$, and hence \( \mathcal H_{K_{\eta_1}}\hookrightarrow \mathcal H_{K_{\eta_2}} \) contractively by Lemma~\ref{lem:rkhs-embedding-main2}.

We next compute $d_\eta(g,h)^2$. Using the definition of $K_\eta$, we have
\begin{align*}
K_\eta(g,g) &= \int_{\widehat H} \theta(0)\,m_\eta(\lambda_H(\theta))\,d\mu_{\widehat H}(\theta) = \int_{\widehat H} m_\eta(\lambda_H(\theta))\,d\mu_{\widehat H}(\theta), \\ 
K_\eta(h,h) &= \int_{\widehat H} m_\eta(\lambda_H(\theta))\,d\mu_{\widehat H}(\theta), \\
K_\eta(g,h) &= \int_{\widehat H} \theta(h-g)\,m_\eta(\lambda_H(\theta))\,d\mu_{\widehat H}(\theta).
\end{align*}
Therefore,
\begin{align*}
d_\eta(g,h)^2 &= K_\eta(g,g)+K_\eta(h,h)-2\Re K_\eta(g,h) \\
&= \int_{\widehat H} \bigl(2-2\Re \theta(h-g)\bigr)\,m_\eta(\lambda_H(\theta))\,d\mu_{\widehat H}(\theta) \\
&= \int_{\widehat H} |\theta(h-g)-1|^2\,m_\eta(\lambda_H(\theta))\,d\mu_{\widehat H}(\theta).
\end{align*}

Since Lemma~\ref{lem:multiplier-order-main} gives \( m_{\eta_1}(\lambda)\le m_{\eta_2}(\lambda)\) for all \(\lambda\ge 0, \) it follows that
\[
d_{\eta_1}(g,h)^2 \le d_{\eta_2}(g,h)^2,
\]
and hence $d_{\eta_1}(g,h)\le d_{\eta_2}(g,h)$. Similarly, from the definition of $B_\eta$ and the same pointwise inequality, \( B_{\eta_1}\le B_{\eta_2}. \)

If $A_{\eta_2}<\infty$, then again by Lemma~\ref{lem:multiplier-order-main}, \( \lambda_H(\theta)m_{\eta_1}(\lambda_H(\theta)) \le \lambda_H(\theta)m_{\eta_2}(\lambda_H(\theta)) \) for all \(\theta \in \widehat H, \)
and therefore $A_{\eta_1}<\infty$ and $A_{\eta_1}\le A_{\eta_2}$.

Finally, using the representation of $d_{\eta_2}$ above with the estimate \( |\theta(h-g)-1|^2\le \lambda_H(\theta)\rho(g,h)^2 \) from Lemma~\ref{lem:char-lip-comparison-uniform}, we obtain
\begin{align*}
d_{\eta_2}(g,h)^2 &\le \rho(g,h)^2 \int_{\widehat H} \lambda_H(\theta)\,m_{\eta_2}(\lambda_H(\theta))\,d\mu_{\widehat H}(\theta) \\
&= A_{\eta_2}\rho(g,h)^2.
\end{align*}
Taking square roots gives \( d_{\eta_2}(g,h)\le A_{\eta_2}^{1/2}\rho(g,h), \) which completes the proof.
\end{proof}

\section{Example}
\label{sec:example}

We show the heat flow construction and the resulting random--walk invariants on a concrete pair of finite weighted graphs. The graphs in Figure~\ref{fig:graphs} are Watts--Strogatz small--world networks \cite{WattsStrogatz1998} with edge weights in $\mathbb N$. The resulting birth and death times lie in the integer lattice above the diagonal, giving a uniformly discrete birth--death geometry and placing the associated virtual persistence diagram group in the discrete locally compact abelian case.

\begin{figure}[t]
\centering
\includegraphics[width=0.48\textwidth]{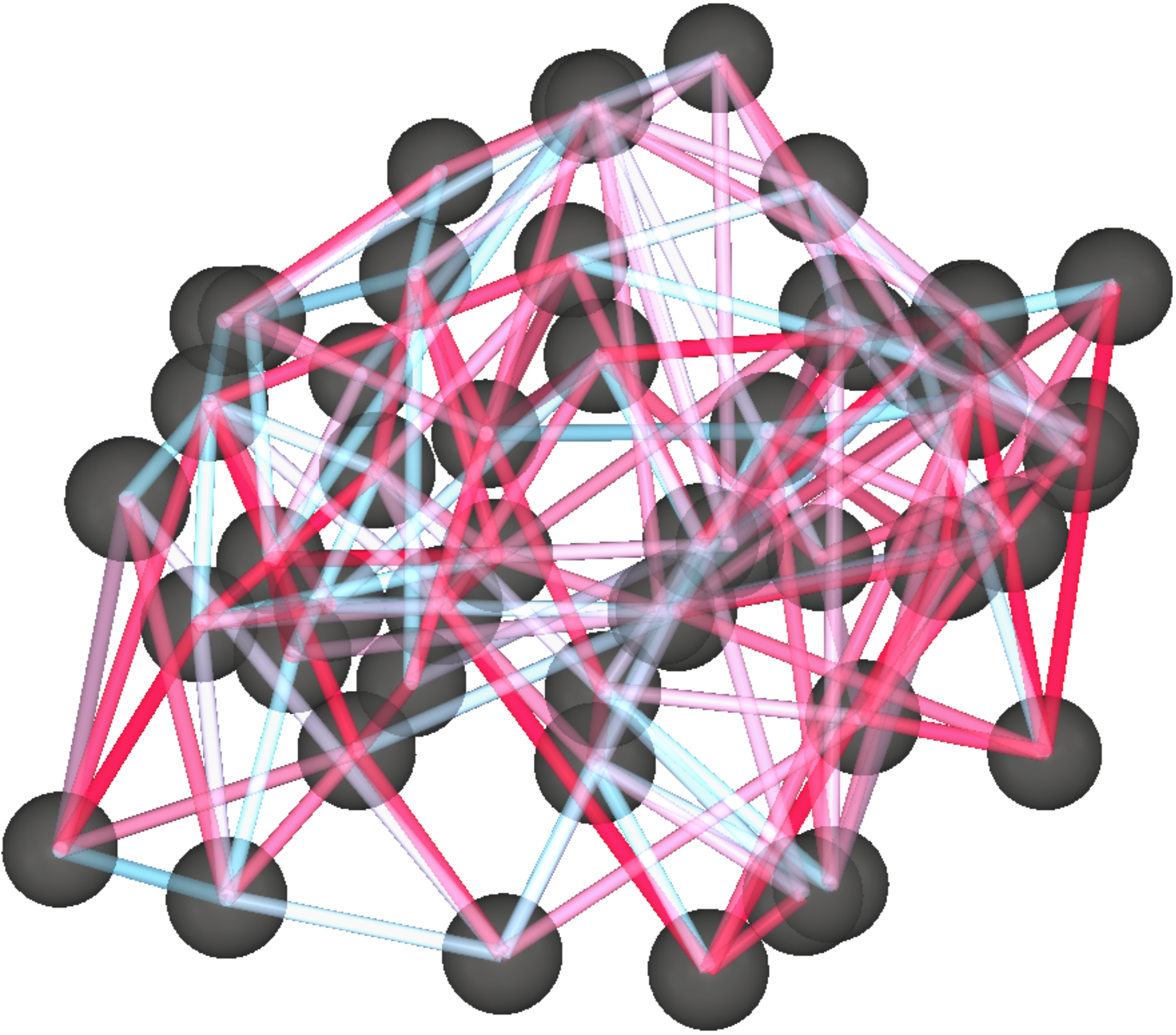}\hfill
\includegraphics[width=0.48\textwidth]{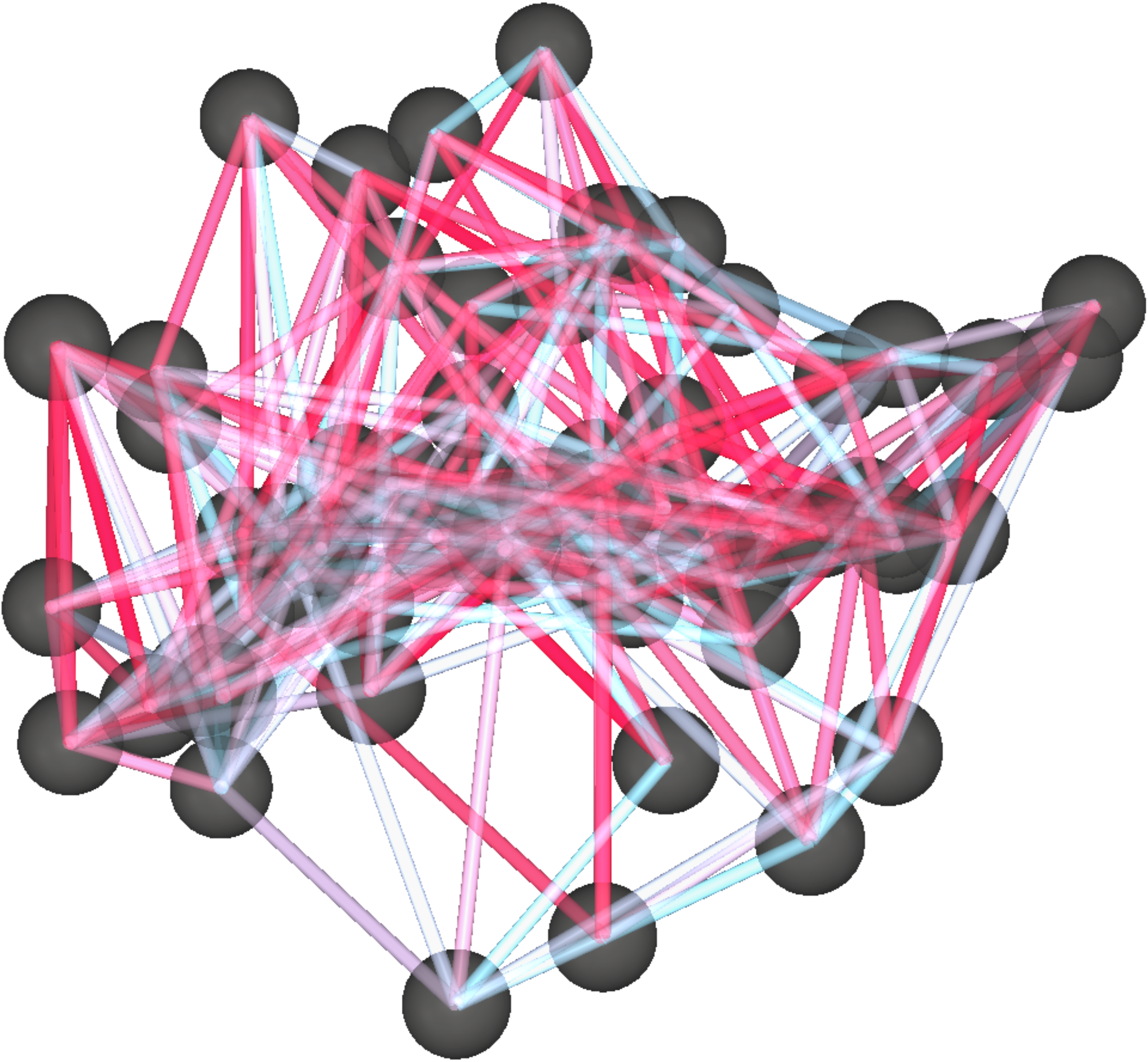}
\caption{Weighted graphs used to generate the persistence diagrams.}
\label{fig:graphs}
\end{figure}

Figure~\ref{fig:graphs} shows the two input graphs. The graph on the left has $50$ vertices, an underlying regular degree parameter $k=6$, rewiring probability $0.3$, and edge weights in the integer interval $[1,8]$. The graph on the right has $60$ vertices, degree parameter $k=8$, rewiring probability $0.4$, and edge weights in $[1,10]$. In both cases the filtration parameter is the edge weight, so every simplex enters at an integer time and the ambient birth--death space is countably infinite.

\begin{figure}[t]
\centering
\includegraphics[width=0.48\textwidth]{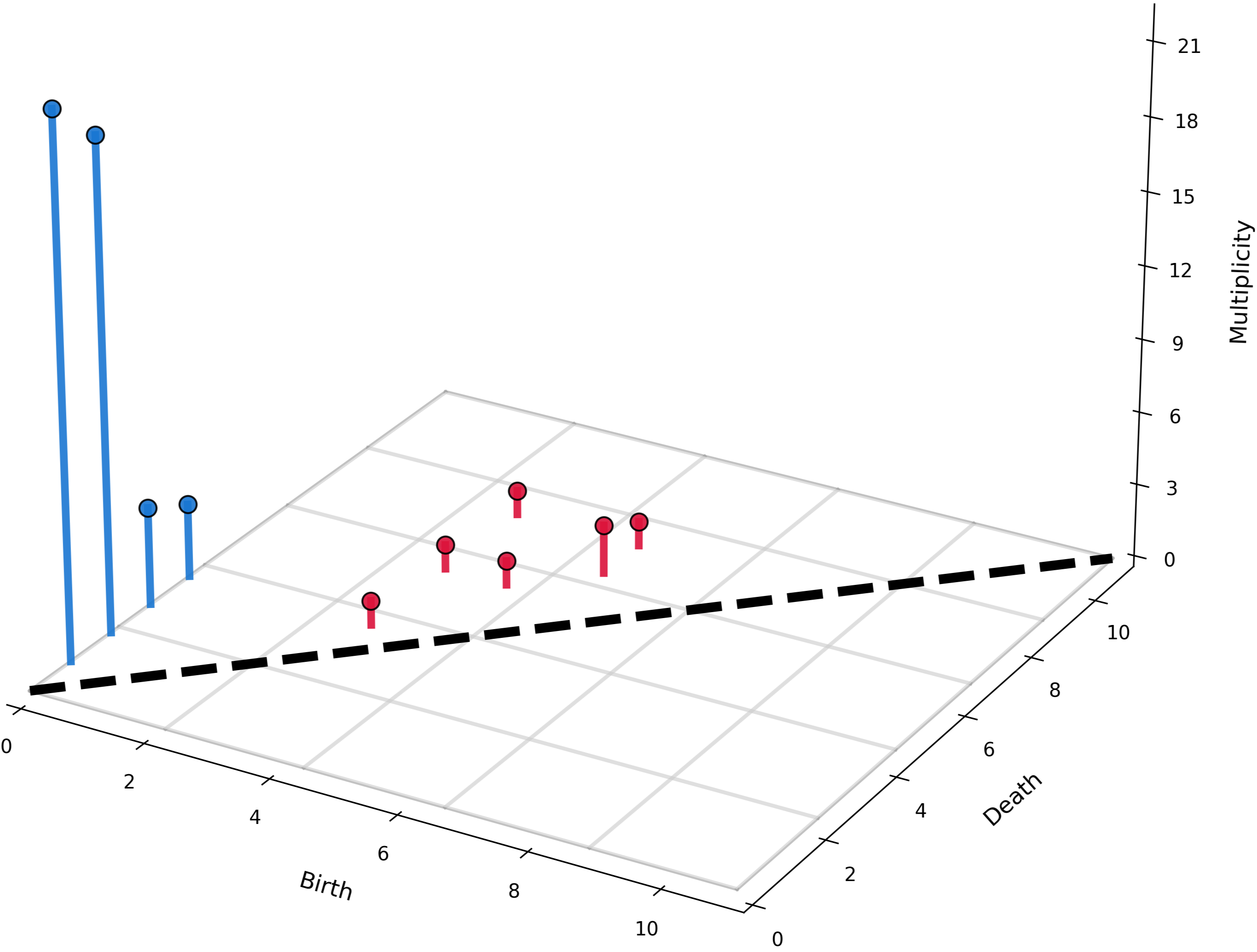}\hfill
\includegraphics[width=0.48\textwidth]{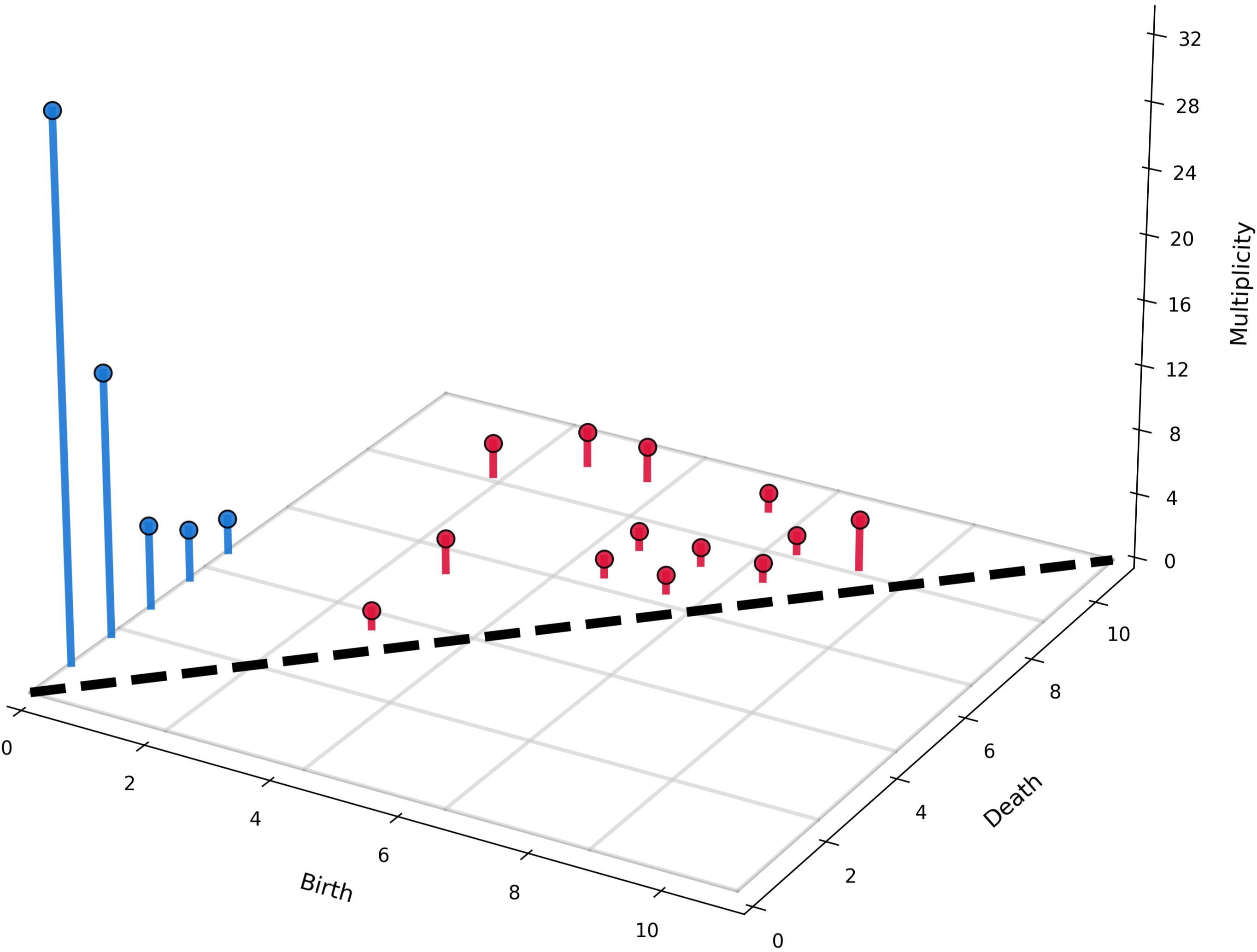}
\caption{The persistence diagrams of the weighted graphs in Figure~\ref{fig:graphs}.}
\label{fig:pds}
\end{figure}

The persistence diagrams shown in Figure~\ref{fig:pds} are supported on a uniformly discrete subset of $X/A$. Their difference defines an element of the virtual persistence diagram group $K(X,A)$ (Figure~\ref{fig:vpd}).

\begin{figure}[t]
\centering
\includegraphics[width=0.6\textwidth]{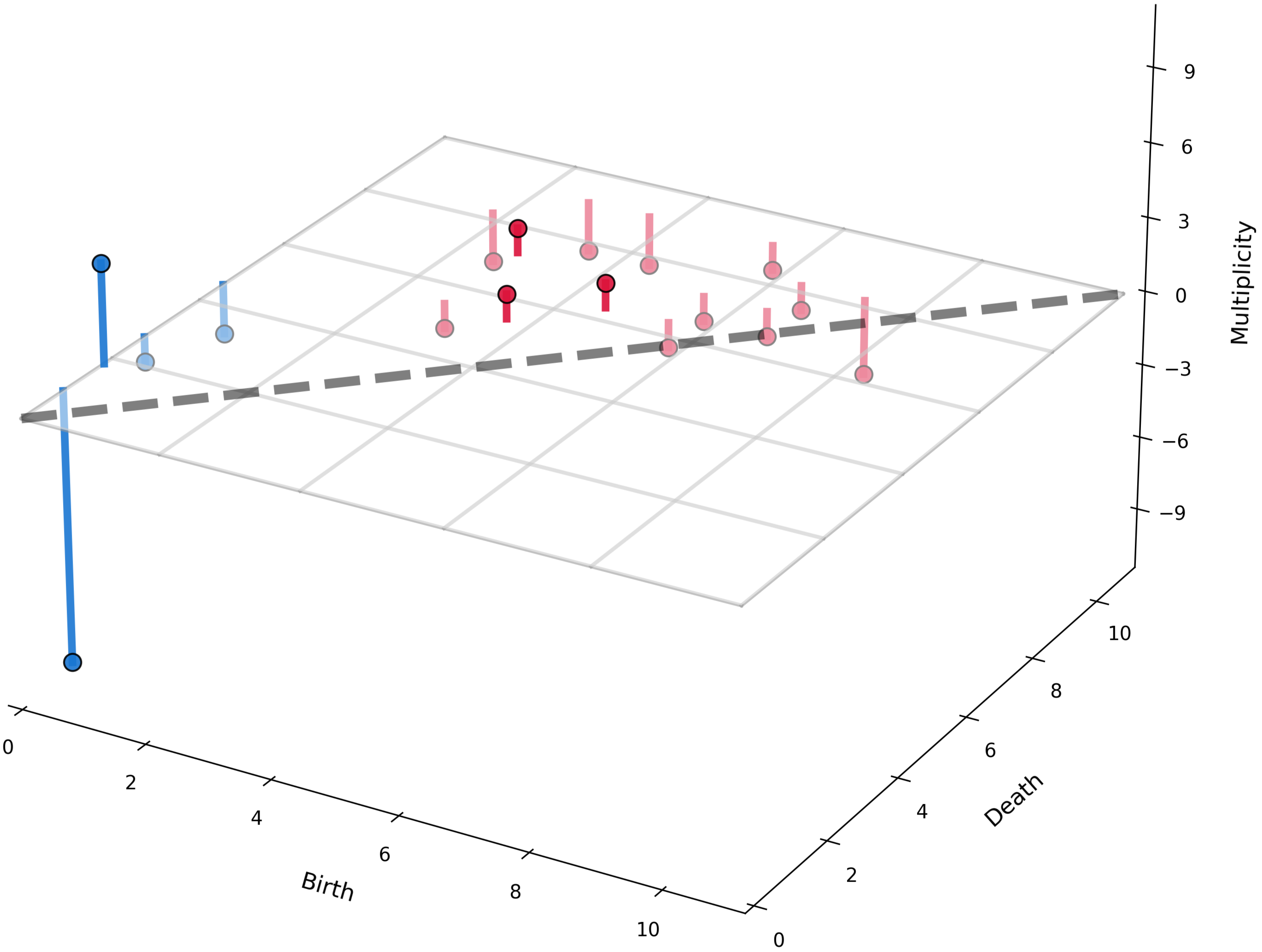}
\caption{The virtual persistence diagram associated with the two persistence diagrams in Figure~\ref{fig:pds}.}
\label{fig:vpd}
\end{figure}

For the plots in Figure~\ref{fig:invariants-bounds}, we compute the negative definite function $\lambda_H$ and the scalar invariants from Subsection~\ref{subsec:heat-invariants} on the subgroup of $H$ generated by the finitely many coordinates that appear in the virtual persistence diagrams used in the example. The shaded bounds in the right panel are obtained by inserting these invariants into the general Lipschitz, $\ell^2(H)\!\to\!\ell^\infty(H)$, and resolvent inequalities for this group $K(X,A)$.

\begin{figure}[t]
\centering
\includegraphics[width=0.48\textwidth]{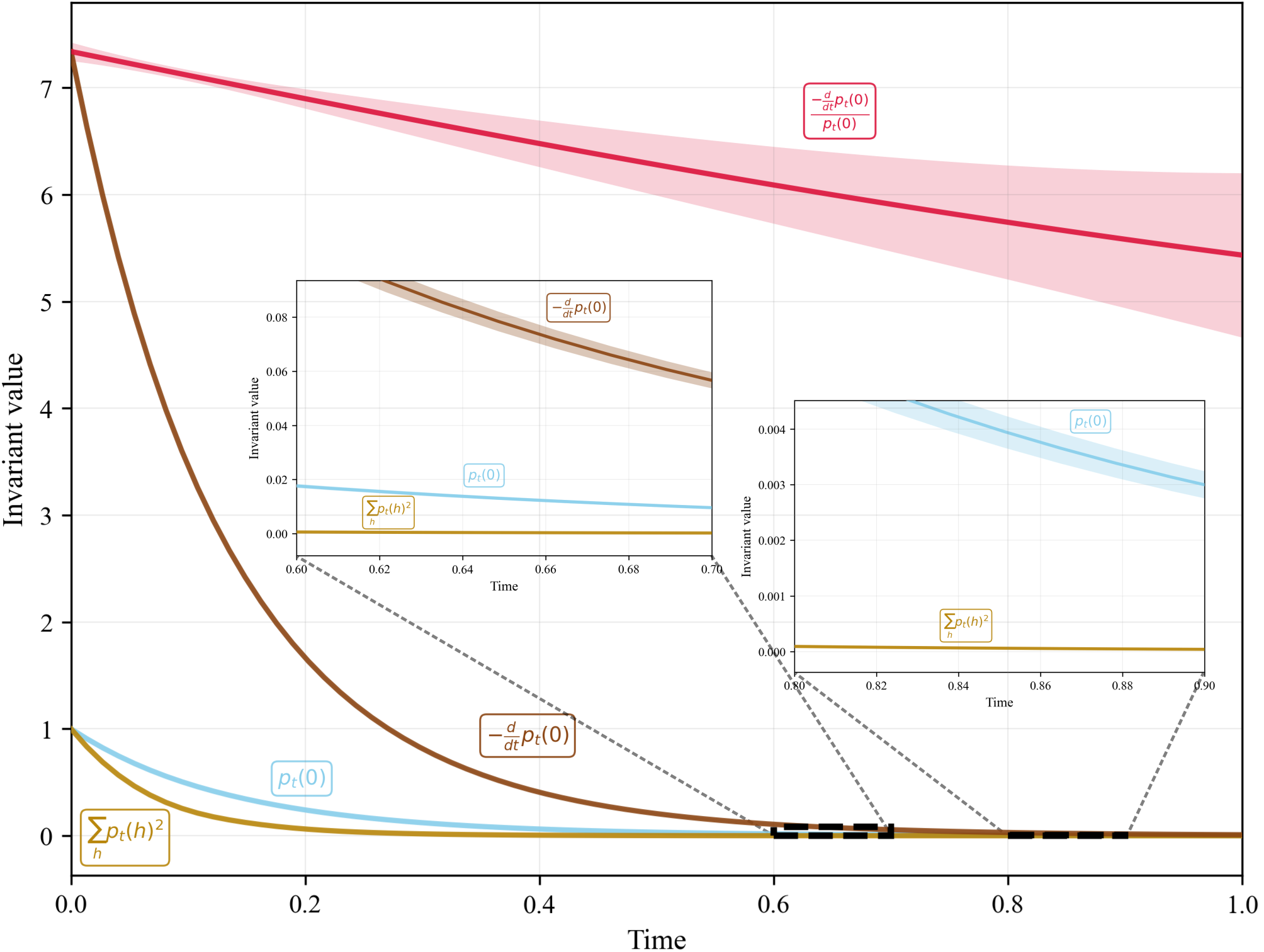}\hfill
\includegraphics[width=0.48\textwidth]{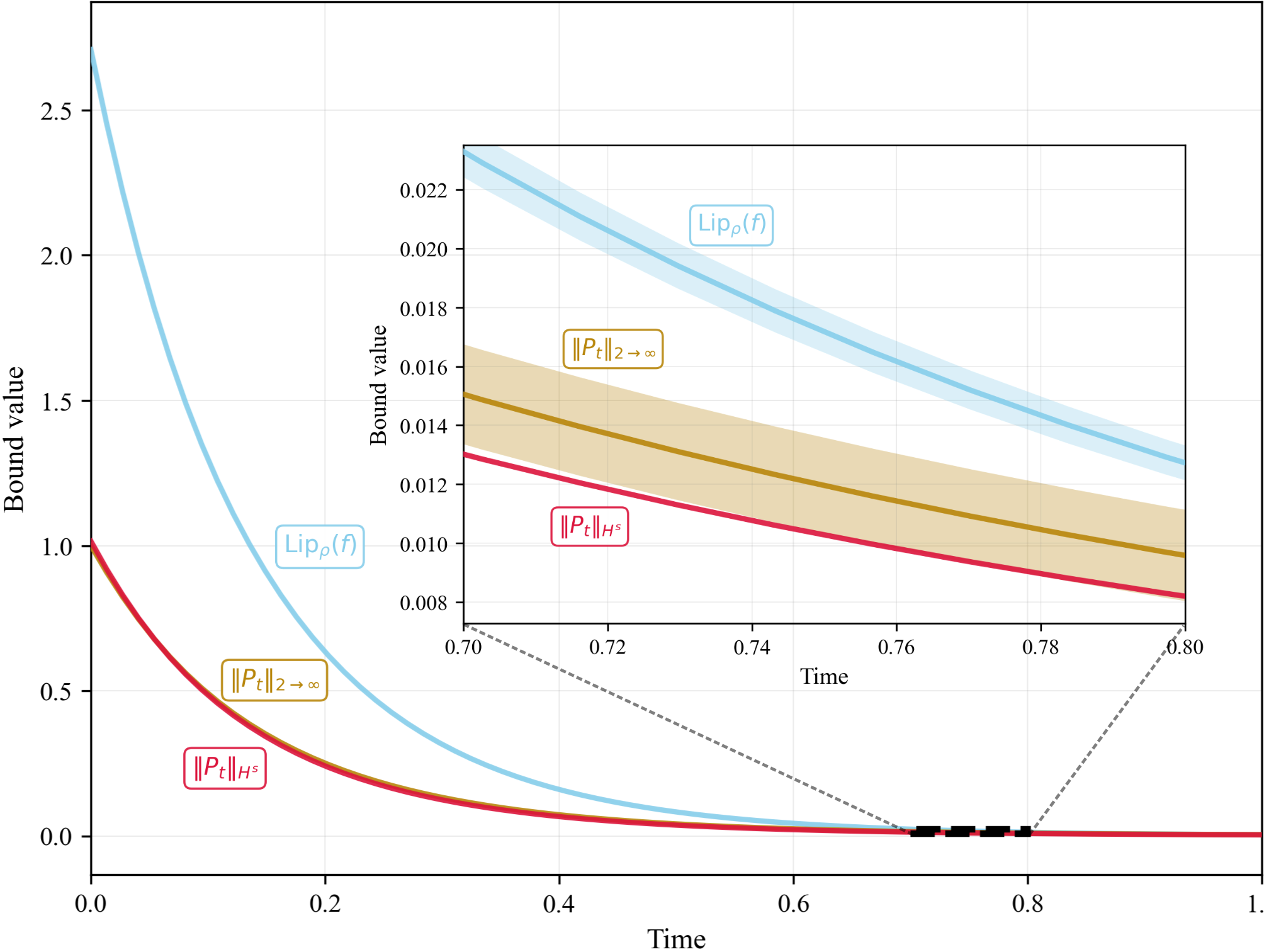}
\caption{Heat kernel invariants and induced bounds as functions of diffusion time $\tau$.}
\label{fig:invariants-bounds}
\end{figure}

Figure~\ref{fig:invariants-bounds} shows the heat kernel invariants and the corresponding bounds associated with the random walk. In the left panel, the four invariants separate at small diffusion time into two quantities of larger magnitude and two of smaller magnitude. As diffusion time increases, three of the four invariants decay monotonically on comparable time scales, while the normalized spectral scale varies much more slowly. The decaying invariants arise as Laplace transforms of the spectral distribution of the generator, whereas the normalized scale is a spectral average and is therefore much less sensitive to uniform exponential damping.

The right panel displays the corresponding bounds obtained by inserting these invariants into the general inequalities proved above, evaluated on the heat kernel section at the identity. The qualitative behavior of the bounds mirrors that of the invariants shown on the left. Bounds that depend multiplicatively on the return probability decrease quickly as diffusion time increases, while the bound controlled by the resolvent varies on a longer scale, reflecting the persistence of low--frequency spectral mass.

\section{Conclusion}

This paper extends the heat kernel and reproducing kernel Hilbert space theory for virtual persistence diagrams from the finite--rank case of \cite{fanning2025reproducingkernelhilbertspaces} to uniformly discrete metric pairs $(X,d,A)$. In this case, we construct a symmetric, translation--invariant heat semigroup $(P_t)_{t\ge0}$ on the virtual persistence diagram group $K(X,A)$ from a summable symmetric pair--jump kernel. We show that this semigroup is supported on a countable subgroup $H \subset K(X,A)$ and carry out the harmonic analysis on $H$. On $H$, the semigroup has a L\'evy--Khintchine exponent $\lambda_H$, which underlies the heat kernels and the kernels, semimetrics, and scalar invariants developed in the paper.

The random walk on $H$ defines four concrete invariants: the return probability, collision probability, heat-kernel energy, and diagonal resolvent value. Each admits both a probabilistic interpretation through the induced L\'evy process and a spectral representation through the symbol $\lambda_H$. These quantities determine the analytic estimates of the theory. The return and collision probabilities give the evaluation and $\ell^2\!\to\!\ell^\infty$ norms, while the heat-kernel energy and diagonal resolvent value give the sharp constants in the Lipschitz and Sobolev-type inequalities; the same quantities control the tail and covering bounds. Truncation of the L\'evy measure produces finite-activity compound-Poisson models whose invariants converge monotonically to the full values, and this convergence gives explicit approximations of the corresponding constants. Section~\ref{sec:stoch-mc} concludes with the heat-scale majorization theorem, which replaces the pure heat weights $e^{-t\lambda_H}$ by mixtures and establishes a convex-order relation that transfers directly to the associated kernels.

The principal limitation of the present theory is its reliance on uniform discreteness of the pointed metric space $(X/A,\overline d_1,[A])$. This assumption is necessary for discreteness and local compactness of the virtual persistence diagram group, for the reduction to the countable subgroup $H$, and for the use of classical harmonic analysis and L\'evy--Khintchine theory. As a consequence, the results do not apply to non--uniformly discrete cases such as the classical birth--death plane over $\mathbb R$, the canonical choice in many applications of persistent homology.

Future work includes examining the correspondence between Gaussian measures and reproducing kernel Hilbert spaces to construct Gaussian process models on $H$, with the heat kernel RKHSs as Cameron--Martin spaces. On the statistical side, these invariants motivate procedures for comparison and inference based on profiles of return, collision, and resolvent quantities. On the analytic side, the Dirichlet form and generator associated with the induced semigroup give a natural starting point for extending the present regularity and approximation theory beyond RKHSs to Sobolev-- and H\"older--type function spaces defined via the semigroup or the VPD metric.

\section*{Statements and Declarations}

\begin{itemize}

\item \textbf{Competing Interests} The authors declare that they have no competing interests.

\item \textbf{Funding} This research received no external funding.

\item \textbf{Data Availability} Not applicable.

\item \textbf{Code Availability} The implementation used in this work is available at \url{https://github.com/cfanning8/Random_Walks_on_Virtual_Persistence_Diagrams}.

\item \textbf{Authors' Contributions}
C.F. developed the theory, conducted and analyzed the examples, and wrote the manuscript. M.E.A. advised the project and gave feedback on the theory and manuscript.

\end{itemize}


\end{document}